\documentclass[aap,preprint]{imsart}

\RequirePackage[OT1]{fontenc}
\RequirePackage{amsthm,amsmath}
\RequirePackage[numbers]{natbib}
\RequirePackage[colorlinks,citecolor=blue,urlcolor=blue]{hyperref}
\usepackage{amsfonts}
\usepackage{mathrsfs}

\startlocaldefs
\newtheorem{thm}{Theorem}[section]

\newtheorem{prop}[thm]{Proposition}

\newtheorem{cor}[thm]{Corollary}

\newtheorem{remark}[thm]{Remark}

\newtheorem{lemma}[thm]{Lemma}
\newtheorem{theorem}[thm]{Theorem}
\newtheorem{corollary}[thm]{Corollary}
\newtheorem{definition}[thm]{Definition}
\newtheorem{proposition}[thm]{Proposition}
\newtheorem{example}[thm]{Example}
\newtheorem{condition}[thm]{Condition}

\newcommand{\bR}{\mathbb{R}}

\newcommand{\bE}{\mathbb{E}}

\newcommand{\bP}{\mathbb{P}}

\newcommand{\al}{\alpha}

\newcommand{\de}{\delta}

\newcommand{\la}{\lambda}

\newcommand{\ep}{\varepsilon}

\newcommand{\goto}{\rightarrow}

\def\<{\left<}\def\>{\right>}
\def\({\left(}\def\){\right)}

\newcommand{\dd}{\mathrm{d}}

\def\blemma{\begin{lemma}\sl{}\def\elemma{\end{lemma}}}
\def\btheorem{\begin{theorem}\sl{}\def\etheorem{\end{theorem}}}
\def\bcorollary{\begin{corollary}\sl{}\def\ecorollary{\end{corollary}}}

\def\bproposition{\begin{proposition}\sl{}\def\eproposition{\end{proposition}}}
\def\bcondition{\begin{condition}\sl{}\def\econdition{\end{condition}}}
\def\bexample{\begin{example}\sl{}\def\eexample{\end{example}}}
\def\bremark{\begin{remark}\sl{}\def\eremark{\end{remark}}}

\def\beqlb{\begin{eqnarray}}\def\eeqlb{\end{eqnarray}}
\def\beqnn{\begin{eqnarray*}}\def\eeqnn{\end{eqnarray*}}

\def\qqquad{\qquad\qquad}
\def\qed{\hfill$\Box$\medskip}

\def\<{\langle}\def\>{\rangle}

\def\e{{\mbox{\rm e}}}

 \def\proof{\noindent{\it Proof.~~}}\def\qed{\hfill$\Box$\medskip}

 \def\<{\langle}\def\>{\rangle}

 \def\mcr{\mathscr}\def\mbb{\mathbb}\def\mbf{\mathbf}

 \def\ar{\!\!&}

 \def\red{\color{red}}
 \def\blue{\color{blue}}

\endlocaldefs

\begin{document}

\begin{frontmatter}
\title{ A general continuous-state nonlinear branching process}
\runtitle{Nonlinear branching process}

\begin{aug}
\author{\fnms{Pei-Sen} \snm{Li}\thanksref{t1}\ead[label=e1]{peisenli@ruc.edu.cn}},
\author{\fnms{Xu} \snm{Yang}\thanksref{t2}\ead[label=e2]{xuyang@mail.bnu.edu.cn}}
\and
\author{\fnms{Xiaowen} \snm{Zhou}\thanksref{t3} \thanksref{t4}
\ead[label=e3]{xiaowen.zhou@concordia.ca}}

\thankstext{t1}{Supported by NSERC (RGPIN-2016-06704) and  NSFC (No. 11771046 and 11571043)}
\thankstext{t2}{Supported by NSFC (No. 11771018 and No. 11401012), NSF of Ningxia (No. 2018AAC03245) and First-Class Disciplines Foundation of Ningxia (No. NXYLXK2017B09)}
\thankstext{t3}{Supported by NSERC (RGPIN-2016-06704) and NSFC (No. 11731012)}
\thankstext{t4}{Corresponding author.}
\runauthor{Li, P.S., Yang, X. and Zhou, X.}

\affiliation{Renmin University of China\thanksmark{t1}, North Minzu University\thanksmark{t2} and Concordia University\thanksmark{t3}}
\address{Address of the First and Second authors
\\
\\
\\
\printead{e1}\\
\phantom{E-mail:\ }\printead*{e2}}

\address{Address of the Third author\\
\\
\\
\printead{e3}\\
}
\end{aug}

\begin{abstract}
In this paper we consider
 the unique nonnegative  solution to the  following generalized version of the stochastic differential equation  for  a continuous-state branching process.
\beqnn
X_t
\ar=\ar
x+\int_0^t\gamma_0(X_s)\dd s+\int_0^t\int_0^{\gamma_1(X_{s-})}
W(\dd s,\dd u)\cr
\ar\ar\qquad+\int_0^t\int_{0}^\infty\int_0^{\gamma_2(X_{s-})} z\tilde{N}(\dd s, \dd z, \dd u),
\eeqnn
where $W(\dd t,\dd u) $ and $\tilde{N}(\dd s, \dd z, \dd u)$ denote a Gaussian white noise and  an independent compensated spectrally positive Poisson random measure, respectively,  and $\gamma_0,\gamma_1$ and $\gamma_2$ are functions on $\mbb{R}_+$ with both $\gamma_1$ and $\gamma_2$ taking nonnegative values.
 Intuitively, this process can be identified as a continuous-state branching process with population-size-dependent branching rates and with competition. Using martingale techniques we find  rather sharp conditions on extinction, explosion  and coming down from infinity
 behaviors of the process.  Some Foster-Lyapunov type criteria are also developed for such a process.
 More explicit results are obtained when $\gamma_i, i=0, 1, 2$ are power functions.
\end{abstract}

\begin{keyword}[class=MSC]
\kwd[Primary ]{60G57}
\kwd[; secondary ]{60G17}\kwd{60J80}
\end{keyword}

\begin{keyword}
\kwd{continuous-state branching process}
\kwd{nonlinear branching}
\kwd{competition} \kwd{extinction} \kwd{explosion} \kwd{coming down from infinity} \kwd{weighted  total population} \kwd{Foster-Lyapunov  criterion} \kwd{stochastic differential equation}
\end{keyword}

\end{frontmatter}

\section{Introduction}
\setcounter{equation}{0}

\subsection{Continuous-state branching processes}

Suppose that $(\Omega,\mcr{F},\mcr{F}_t,\mbb{P})$ is a filtered probability
space satisfying the usual hypotheses.
Let $\bP_x$ be the law of a process
started at $x$, and denote by $\bE_x$ the associated expectation.
A continuous-state branching process  $X= (X_t)_{t\ge 0}$ is
a  c\`adl\`ag $[0,\infty]$-valued $(\mcr{F}_t)$-adapted process satisfying the branching property, i.e. for any $x, y\geq 0$ and $t, \theta \geq 0$,
\begin{equation}\label{branching_prop}
\bE_{x+y} \big[\e^{-\theta X_t}\big]=\bE_x \big[\e^{-\theta X_t}\big]\bE_y \big[\e^{-\theta X_t}\big].
\end{equation}
Consequently, its Laplace transform is determined by
 \beqnn
\bE_x \big[\e^{-\theta X_t}\big]=\e^{-xu_t(\theta)},
 \eeqnn
where the non-negative function $u_t(\theta)$ solves  the differential equation
 \beqnn
\frac{\partial u_t(\theta)}{\partial t}+\psi(u_t(\theta))=0\
 \eeqnn
with initial value $u_0(\theta)=\theta\ge0$ and Laplace exponent
 \beqnn
\psi(\la)=b\la+\frac{1}{2}\sigma^2\la^2+\int_0^\infty (\e^{-\la x}-1+\la x)\pi(\dd x)
 \eeqnn
for $b\in \bR,\sigma\geq 0$ and for $\sigma$-finite measure
$\pi$ on $(0, \infty)$ satisfying $\int_0^\infty (z\wedge
z^2)\pi(\dd z)<\infty $.

Via the Lamperti random time change the continuous-state branching process is associated to a spectrally positive L\'evy process, which allows many semi-explicit expressions. In particular, extinction and explosion behaviors for continuous-state branching processes were studied by Grey (1974) and Kawazu and
Watanabe (1971), respectively, and  the conditions for extinction and explosion were expressed in terms of the respective integral tests on the function $\psi$.

Bertoin and Le Gall (2006) and Dawson and Li (2006, 2012) noticed the following
alternative way of characterizing continuous-state branching processes  through
stochastic differential equations (SDEs in short).
Let $\{W(\dd t,\dd u): t,u\ge 0\}$ denote an $(\mcr{F}_t)$-Gaussian white noise
with density measure $\dd t\dd u$ on $(0,\infty)^2 $.
In this paper we always write $\pi\neq0$ for a $\sigma$-finite measure on $(0, \infty)$.
Let $\{N(\dd t,\dd z,\dd u):t,z,u>0\}$ denote an independent $(\mcr{F}_t)$-Poisson
random measure with intensity measure $\dd t\pi(\dd z)\dd u$ on $(0,\infty)^3$ and  let
$\{\tilde{N}(\dd t,\dd z,\dd u):t,z,u>0\}$ denote the corresponding compensated measure.
Then the continuous-state branching process is
a pathwise unique nonnegative solution to the following SDE that is called a Dawson-Li SDE in Pardoux (2016):
 \beqlb\label{sdeA}
X_t
\ar=\ar
x+b\int_0^t X_s\dd s+\sigma\int_0^t\int_0^{X_{s-}}
W(\dd s,\dd u)\cr
\ar\ar\qquad+\int_0^t\int_0^\infty\int_0^{X_{s-}} z\tilde{N}(\dd s, \dd z, \dd u).
 \eeqlb
SDEs similar to (\ref{sdeA}) were studied by Dawson and Li (2006, 2012) and by
Fu and Li (2010). See also Bertoin and Le Gall (2003, 2005) for related work.

We refer to Kyprianou (2006),
Li (2011, 2012) and Pardoux (2016) for reviews and literature on continuous-state branching processes.

\subsection{Continuous-state  branching processes with nonlinear branching}\label{nonlinear}

Models with interactions have  gained interests in the study of branching processes. Athreya and Ney (1972) introduced
 population-size-dependent Galton-Watson branching processes in which the reproduction mechanism depends on the population size; see also
 Klebaner (1984) and H\"opfner (1985) for previous work on population-size-dependent Galton-Watson processes. Another class of interacting Galton-Watson processes is the so called controlled branching processes. For a controlled branching process, the reproduction law is fixed. But before each branching time the population is regulated by a control function.
 Previous work on controlled branching processes can be found in Sevast'yanov and Zubkov (1974) and references therein.  A discrete state, continuous time branching process with population dependent branching rate can be found in  Chen (1997).
 When the branching rate function  is a power function of the population, the extinction probability for such a branching process  was obtained in Chen (2002).  When the branching rate is a general positive nonlinear function, such a model called nonlinear Markov branching process was studied in Pakes (2007).

The previous work on discrete-state interacting branching processes motivates the study of their continuous-state counterparts. Some population-size-dependent continuous-state branching processes arising as scaling limits of the corresponding discrete-state branching processes can be found in Li (2006, 2009).

In this paper we introduce a class of continuous-state branching processes whose branching rates depend on their current population sizes. To this end, we  consider a nonnegative solution to the following modification of SDE (\ref{sdeA}):
\beqlb\label{sdeB}
X_t
\ar=\ar
x+\int_0^t\gamma_0(X_s)\dd s+\int_0^t\int_0^{\gamma_1(X_{s-})}
W(\dd s,\dd u)\cr
\ar\ar\qquad+\int_0^t\int_{0}^\infty\int_0^{\gamma_2(X_{s-})} z\tilde{N}(\dd s, \dd z, \dd u),
\eeqlb
where $\gamma_0,\gamma_1$ and $\gamma_2$ are Borel functions on $\mbb{R}_+$, and both $\gamma_1$ and $\gamma_2$ take nonnegative values.
The unique nonnegative solution to (\ref{sdeB}) up to the minimum of its first time of hitting $0$ and its explosion time can be treated as a continuous-state nonlinear branching process,  where $\gamma_i(x)/x, i=1,2$ can be interpreted as population-size-dependent branching rates and the drift term involving $\gamma_0$ can be related to either competition or population-size-dependent continuous immigration. We refer to Duhalde \textit{et al.} (2014) for work on continuous-state branching processes with immigration.
If $\gamma_i(x)=c_i x$ for $c_1,c_2\ge0,$ then the solution to (\ref{sdeB}) reduces to the classical continuous-state branching process and satisfies the branching property (\ref{branching_prop}).
Observe that the solution $X$ to (\ref{sdeB}) can also be treated as a continuous-state controlled branching process.


For  $\gamma_2\equiv 0,  \gamma_1(z)=z$ and $\gamma_0$ satisfying certain conditions, the SDE
(\ref{sdeB}) was studied in Pardoux and Wakolbinger (2015) and in Pardoux (2016) where
the function $\gamma_0$ models an impact of the current population
size on the individuals' reproduction dynamics.
If the interaction is of the type of competition for
rare resources, then  increasing  the population size results in a reduction of the individuals' birth
rate and/or increment of the death rate.

For $ \gamma_1(z)=\gamma_2(z)=z$ and $\gamma_0(z)=\theta z-\gamma z^2$
with positive constants $\theta$ and $\gamma$,
solution  to SDE \eqref{sdeB} can be used to model the density dependence in population dynamics of
a large population with competition called logistic branching process,
and it was studied in detail by Lambert (2005).
The quadratic regulatory term has an ecological interpretation  as
it describes negative interactions between each pair of individuals in the population.
The extinction behavior and the probability distribution of the extinction time were considered in Lambert (2005).
A similar model with more general function $\gamma_0 $  was considered in  Le \textit{et al.} (2013) with its first passage times studied. The total mass for this model was also studied using the Lamperti transform.  Berestycki et al. (2017) gave a genealogical description for the process based on interactive pruning of L\'{e}vy--trees, and established a Ray--Knight representation result.

For $\gamma_0(z)=\gamma_2(z)\equiv0$,  the extinction/survival behaviors for process $X$ as the total mass process of a superprocess with mean field interaction were discussed in Wang \textit{et al.} (2017) by a martingale approach.  More generally, for $\gamma_2(z)\equiv0$ the extinction, explosion and coming down from infinity behaviors for diffusion process $X$ are associated to    the classification of its boundaries at $0$ and $\infty$, respectively;  see Karlin and Taylor (1981, p. 229).

For $\gamma_i(z)=c_iz^r$
with $r>0$, $c_0\in\mbb{R}$ and $c_i\ge0$ for $i=1,2$, the solution to SDE \eqref{sdeB},
 called a continuous-state polynomial branching process, was studied by Li (2018),
 where the parameter $r$ describes the degree of interaction. The polynomial branching process also arises as time-space scaling limit of discrete-state nonlinear branching processes.  Intuitively, functions  $\gamma_1$ and $\gamma_2$ are population-dependent rates for branching events producing small and large amount of children, respectively. By solving the corresponding Kolmogorov equations,  necessary and sufficient conditions in terms of integral tests were obtained for extinction, explosion and coming down from infinity, respectively. Expectations of the extinction time and explosion time were also discussed in Li (2018), which generalizes those results in Chen (2002) for discrete-state processes to the corresponding continuous-state processes.
 The nonlinear branching processes considered in this paper generalize those in Li (2018) by allowing different rates for different branching events.


Note that if $\tilde{N}$ is the compensated measure of a
one-sided $\alpha$-stable random measure with $\alpha\in(1,2)$,
i.e.,
 \beqlb\label{pi}
\pi(\dd z)=\frac{\alpha(\alpha-1)}{\Gamma(2-\alpha)}
1_{\{z>0\}}z^{-1-\alpha}\dd z
 \eeqlb
for  Gamma function $\Gamma$, then on an enlarged probability space,
SDE \eqref{sdeB} can be transformed
into the following SDE:
 \beqlb\label{sdeD}
X_t\ar=\ar x+\int_0^t\gamma_0(X_s)\dd s+\int_0^t
\sqrt{\gamma_1(X_{s-})}\dd B_s\cr
\ar\ar\qquad+\int_0^t \gamma_2(X_{s-})^{1/\alpha}
\int_{0}^\infty u\tilde{M}(\dd s, \dd u),
 \eeqlb
where $\{B_t:t\ge0\}$ is a Brownian motion and $\{\tilde{M}(\dd t, \dd u):t,u\ge0\}$ is an independent
compensated Poisson random measure with intensity measure $\dd t\pi(\dd u)$;
see Theorem 9.32 in Li (2011) for a similar result.
Equation (\ref{sdeD}) has a pathwise unique non-negative strong solution
if $\gamma_0(z)=a_1z+a_2$, $\gamma_1(z)=z^{r_1}$ and
$\gamma_1(z)=z^{r_2}$ for $a_1\in\mbb{R},a_2\ge0$,
$r_1\in[1/2,1]$ and $r_2\in(\alpha-1,\alpha]$; see Corollary 4.3 in Li and Mytnik (2011).
By Theorem 4.1.2 in Li (2012) one can also convert (\ref{sdeB}) to another SDE:
\beqnn
X_t=x+\int_0^t\gamma_0(X_s)\dd s
+\int_0^t \sqrt{\gamma_1(X_{s-})}\dd B_s
+\int_0^t\int_{0}^\infty u\tilde{M}_{\gamma_2}(\dd s, \dd u),
\eeqnn
where $\tilde{M}_{\gamma_2}(\dd s, \dd u)$ is  an optional
compensated Poisson  measure with predictable compensator $\gamma_2(X_{s-})\dd s\pi(\dd u)$.

Using the Lamperti transform for positive self-similar Markov processes,
Berestycki \textit{et al.} (2015) found the extinction condition of solution to \eqref{sdeD}
for
$$\gamma_0(z)=\theta z^{\eta}f(z), \gamma_1(z)\equiv0, \gamma_2(z)=z^{\alpha\beta}\,\text{  and} \, \pi(\dd z)=\frac{\alpha(\alpha-1)}{\Gamma(2-\alpha)}
1_{\{z>0\}} z^{-1-\alpha}\dd z$$
with $\alpha\in (1,2), \theta\geq 0, \beta\in [1-1/\alpha, 1),  \eta=1-\alpha(1-\beta)\in[0,1)$ and for certain nonnegative Lipschitz continuous function $f$.
In particular, for $f\equiv 1$ it is shown that the extinction occurs within finite time with probability one for $0\leq \theta< \Gamma(\alpha)$ and with probability $0$ for $\theta \geq \Gamma(\alpha)$; see Theorems 1.1 and 1.4 of Berestycki \textit{et al.} (2015).


We refer to Lambert (2005), Berestycki \textit{et al.} (2010), Bansaye \textit{et al.} (2015) and Li (2018) for previous studies  of coming down from infinity for a branching process with logistic growth, coalescents, birth and death processes and the polynomial branching process, respectively.

 Other than the above mentioned results, we are not aware of any previous results on hitting probability and coming down from infinity for solutions to SDEs of type (\ref{sdeB}). There is some  literature on nonexplosion of solutions to general SDE with jumps; see Dong (2016) for a recent result. But we do not find any systematic discussions on the explosion/nonexplosion dichotomy and the coming down from infinity property of the solutions.


The main purpose of this paper is to
investigate  the extinction, explosion and coming down from infinity  behaviors of the continuous-state nonlinear branching process as solution to (\ref{sdeB})
and specify the associated conditions on functions
$\gamma_i$, $i=0, 1, 2$.

For lack of negative jumps, the extinction behaviors depend on the asymptotic behaviors of function $\gamma_i(x)$ as $x\goto 0+$.
Intuitively, extinction can either be caused by a large enough negative drift due to $\gamma_0$ or large enough fluctuations due to $\gamma_1$ or $\gamma_2$. Even when the process has a (small) positive drift near $0$, it might still die out because of relative large fluctuations.

We are also interested in the relations between the asymptotes of functions $\gamma_i(x), i=0, 1, 2$ as $x\goto\infty$ and the explosion  and coming down from infinity behaviors of the nonlinear branching processes as solutions to SDEs (\ref{sdeB}).

When $\gamma_i$, $i=0, 1, 2$ are not power functions with the same power, the  approach of Li (2018)  fails to work.
To overcome this difficulty we adopt an alternative martingale  approach  that appeared earlier in Wang \textit{et al.} (2017).  Such an approach typically involves understanding how the process exits from consecutive intervals near $0$ with the interval lengths decreasing geometrically,  or consecutive intervals near $\infty$ with the interval lengths  increasing geometrically.   To this end,
we construct the corresponding martingale in each situation. These martingales allow  to obtain estimates on both the sequential exit probabilities and sequential exit times via optional stopping, where the lack of negative jumps for process $X$ comes in handy. The desired results then follow
from Borel-Cantelli type arguments.
 Although we focus on SDEs of type (\ref{sdeB}),  we expect that this approach could also adapted to study  similar properties of solutions to other SDEs with more general jump mechanism, and it remains to be checked how sharp the desired results can be.


 In addition, we show that the general nonlinear branching processes considered in this paper  are closed under
  a Lamperti type transform, which allows us to discuss the finiteness of a weighted occupation time until extinction or explosion of the continuous-state nonlinear branching process via considering the extinction or explosion behaviors of the time changed process.

We also find Foster-Lyapunov type
criteria  to show the irreducibility of  the nonlinear continuous-state branching processes, which is of
independent interest.  We refer to Chen (2004) and
Meyn and Tweedie (1993) for the Foster-Lyapunov type criteria
for explosion and stability of Markov chains.

This paper is structured as follows. After introductions in Subsections 1.1 and 1.2 on the continuous-state branching processes,  Section \ref{main_result} summarizes the main results of this paper with an application and examples, where our results are compared with the known results.  In Section \ref{main} we show that SDE \eqref{sdeB} has a unique strong solution up to the first time of reaching $0$ or explosion given that the  functions $\gamma_i, \, i=0,1,2$ are locally Lipschitz on $(0, \infty)$. Section \ref{Foster_Lyapunov} contains Foster-Lyapunov criteria type results that can be used to show the irreducibility of  the solution as a Markov process.
Proofs of the main results in Section \ref{main_result} are included in Section \ref{proofs}.

\section{Extinction, explosion and coming down from infinity}\label{main_result}
\setcounter{equation}{0}

With the convention $\inf\emptyset:=\infty$, for $y>0$ define
 \beqnn
\tau_y^-\equiv \tau^-(y):=\inf\{t> 0: X_t<y\}, ~~
\tau_y^+\equiv \tau^+(y):=\inf\{t>0: X_t>y\}
 \eeqnn
and
 \beqnn\tau_0^-:=\inf\{t> 0: X_t=0 \}.\eeqnn
By  a solution to SDE  (\ref{sdeB}) we mean  a c\`adl\`ag process
$X=(X_t)_{t\ge0}$ satisfying (\ref{sdeB}) up to time $\tau_n := \tau^-_{1/n}\wedge\tau^+_{n}$ for each $n\ge 1$ and
$X_t = \limsup_{n\to \infty}X_{\tau_n-}$ for $t\ge \tau:= \lim_{n\to \infty} \tau_n$. Then  both of the boundary points $0$ and $\infty$ are absorbing for $X$ by definition.

Throughout this subsection we assume that SDE (\ref{sdeB}) allows a unique weak solution denoted by $X:=(X_t)_{t\ge0}$, and consequently the process $X$ has the strong Markov property. In Theorem \ref{thm00} we are going to show that (\ref{sdeB}) allows a pathwise unique solution if the coefficient functions $\gamma_i, i=0,1,2$ are all locally Lipschitz.
We also assume that either $\gamma_1\not\equiv0$ or $\gamma_2\not\equiv0$
and that the functions $\gamma_0$, $\gamma_1 $ and $\gamma_2$ are all
locally bounded on $[0, \infty)$.

 In the following we present our main results on extinction, explosion and coming down from infinity properties of process $X$. Most of the proofs are deferred to Section \ref{proofs}.

   For $a>0$ and $u>0$, let
 \begin{equation}\label{1.6}
 \begin{split}
H_a(u)&:=\int_{0}^\infty\big[(1+zu^{-1})^{1-a}-1-(1-a)z u^{-1}\big]\pi(\dd z)\\
&=a(a-1)u^{-2}\int_0^\infty z^2\pi(\dd z) \int_0^1 (1+zu^{-1}v)^{-1-a}(1-v)\dd v,
 \end{split}
 \end{equation}
where we use the following form of Taylor's formula that is often needed in the proofs of this paper; see e.g. Zorich (2004, p.364) for its proof.

\blemma\label{t4.1}
If function $g$ has a bounded continuous second derivative on $[0, \infty)$, then for any $y, z>0$ we have
 \beqnn
g(y+z)-g(y)-zg'(y)=z^2\int_0^1 g''(y+zv)(1-v)\dd v.
 \eeqnn
\elemma

Note that for $\pi(\dd z)=cz^{-1-\alpha}$ with $\alpha\in (1, 2)$ and $c>0$,
 \beqlb\label{H_a}
H_a(u)=a(a-1)u^{-\alpha}\int_0^\infty cy^{1-\alpha}\dd y\int_0^1 (1+yv)^{ -1-a}(1-v)\dd v.
 \eeqlb
Put
 \beqlb\label{1.7}
G_a(u)
 \ar:=\ar
(a-1)\gamma_0(u)u^{-1}-2^{-1}a(a-1)u^{-2}\gamma_1(u)
-\gamma_2(u)H_a(u).
 \eeqlb

 We choose the function $G_a$  to be of the particular form in (\ref{1.7}) so that, by Ito's formula,  the process  constructed in Lemma \ref{t3.3} can be shown to be a martingale, which is key  for the main proofs in Section \ref{proofs}. The martingale allows to obtain estimates on  exits times of the  processes $X$ via optional stopping. The conditions for extinction, explosion and coming down from infinity for the process $X$ can be identified from the asymptotic behaviors of $G_a(u)$ for $u$ near $0$ or  near $\infty$.  An earlier version of $G_a$ can be found in Wang \textit{et al.} (2017) where it was also used to construct a continuous martingale to study the extinction behavior for the interacting super-Brownian motion.

\bremark
Suppose that $\pi\neq0$ and $u\in(0,c)$ for some constant $c>0$. One can see that
\begin{itemize}
\item
If there exists a constant $\alpha\in(1,2)$ so that
\beqnn
\sup_{0<y<c}y^{\alpha-2}\int_0^y z^2 \pi(\dd z)\le b,\,\,~~
\sup_{0<y<c}y^{\alpha-1}\int_y^\infty z \pi(\dd z)\le b,\eeqnn
then
 \beqnn
H_a(u)\le \frac{(a-1)(a+2)}{2}bu^{-\al}\mbox{ for }\quad a>1
 \eeqnn
\item
If there exists a constant $\alpha\in(1,2)$ so that
 \beqnn
 \inf_{0<y<c}y^{\alpha-2}\int_0^y z^2 \pi(\dd z)\ge b',
 \eeqnn
then
 \beqnn
-H_a(u)\ge \frac{a(1-a)}{2}b'u^{-\al}
\mbox{ for }\quad 0<a<1.
 \eeqnn
\end{itemize}
\eremark


\subsection{Extinction behaviors}

We first present the two main results on  the extinction behaviors for $X$.  Here we only consider the case that the initial value of $X$ is small. In this way we only have to impose conditions on function $G(u)$ for small positive values of $u$. These results, combined with Foster-Lyapunov criteria (Lemmas \ref{lemA} and \ref{lem1}), can be used to discuss the extinction behaviors for $X$ with arbitrary initial value.

\btheorem\label{t2.5}
\begin{itemize}
\item[{\normalfont(i)}]
Suppose that there exist constants $a>1$ and $r<1$ so that
$G_a(u)\geq -(\ln u^{-1})^{r}$
for all small enough $u>0$.
Then we have $\bP_x\{\tau^-_0<\infty \}=0$ for all  $x>0$.   
\item[{\normalfont(ii)}]
Suppose that there exist constants $0<a<1$ and $r>1$ so that
$G_a(u)\geq
(\ln u^{-1})^{r}$ for all small enough $u>0$. Then
$\bP_x\{\tau^-_0<\infty\}> 0$ for all small enough $x>0$.
\end{itemize}
\etheorem

Proof of Theorem \ref{t2.5} is deferred to Section \ref{proofs}.

The next results concern the first passage probabilities for which we  need the following condition.

\bcondition\label{Condition}
\begin{itemize}
\item[{\normalfont(i)}]
For any  $x$ and $a$ with  $x>a>0$,
 \beqlb\label{Condition_a}
\bP_x\{\tau^-_a<\infty \}>0.
 \eeqlb
\item[{\normalfont(ii)}]
For any  $x$ and $a$ with  $x>a>0$,
 \beqlb\label{2.1}
\bP_x\{\tau^-_a<\infty \}=1.
 \eeqlb
\end{itemize}
\econdition

Proof of the next corollary is deferred to Section \ref{proofs}.

\begin{cor}\label{t1.1}
Suppose that the assumption of Theorem \ref{t2.5} (ii) holds.
Then
\begin{itemize}
\item
$\bP_x\{\tau^-_0<\infty\}> 0$ for all $x>0$ if Condition
\ref{Condition} (i) holds;
\item
$\bP_x\{\tau^-_0<\infty\}=1 $ for all $x>0$ if Condition
\ref{Condition} (ii) holds.
\end{itemize}
\end{cor}

For $a\le b$ define
 \beqnn
\Phi(a,b):=
\inf_{y\in[a,b]} \gamma_1(y) +
\inf_{y\in[a,b]} \gamma_2(y ) 1_{\{\int_0^1z\pi(\dd z)=\infty\}}.
 \eeqnn
We can show that (i) or (ii) of Condition \ref{Condition}  hold under certain conditions on
$\gamma_i, \,\, i=0,1,2$.

\bproposition\label{thmA}
\begin{itemize}
\item[{\normalfont(i)}]
Given $x>a>0$, (\ref{Condition_a}) holds if
$\Phi(a,b)>0$ and $\sup_{a\le y\le b}\gamma_0(y)<\infty$ for all $b>a$.
\item[{\normalfont(ii)}]
Given $x>a>0$, suppose that $\Phi(a,b)>0$ for all $b>a$ and that
$\gamma_0(y)\leq 0$ for all large enough $y$.
Then \eqref{2.1} holds.
\item[{\normalfont(iii)}]
If $\gamma_0(a)\leq 0$ and $\Phi(a,b)>0$ for all  $b\ge a>0$,
then for each $x>0$, $\bP_x$-a.s. \, $X_t\rightarrow 0$ as $t\rightarrow\infty$.   Further, by the strong Markov property  either
\[\bP_x\{X_t=0 \,\,\text{ for all $t$ large enough} \}=1\]
or
\[\bP_x\{X_t\rightarrow 0, \,\, \text{but}\,\, X_t>0 \,\,\text{ for all $t$} \}=1,\]
and we say extinguishing occurs in the latter case.
\end{itemize}
\eproposition

Proof of Proposition \ref{thmA} is deferred to Section \ref{Foster_Lyapunov} after Lemma \ref{lem1}.

\bremark
\begin{itemize}
\item[{\normalfont(i})]
 Combining Proposition \ref{thmA} and Theorem \ref{t2.5} (ii) we find conditions for extinction with probability one and extinguishing with probability one, respectively.
\item[{\normalfont(ii)}]
If $\gamma_0=\gamma_2\equiv0$, then
the process $X$ is the total mass of an interacting super-Brownian motion
and  Theorem \ref{t2.5} generalizes Theorems 3.4 and 3.5 of Wang \textit{et al.} (2017).
\end{itemize}
\eremark

\subsection{Explosion  behaviors}

Let $\tau^+_\infty:=\lim_{n\goto\infty}\tau^+_n$ be the explosion time. The solution $X$ to SDE \eqref{sdeB} explodes at a finite time if $\tau^+_\infty<\infty$.
We now present results on the explosion behaviors for $X$ in the following,  and again, we only consider the case of large initial values.

\begin{thm}\label{t3.1}
\begin{itemize}
\item[{\normalfont(i)}]
If there exist constants $0<a<1$ and $r<1$ so that $G_a(u)\geq -(\ln u)^{r}$ for all $u$ large enough, then
$\bP_x\{\tau^+_\infty<\infty\}=0$  for all $x>0$.
\item[{\normalfont(ii)}]
If there exist $a>1$ and $r>1$ so that $G_a(u)\geq (\ln u)^r$ for all $u$ large enough, then $\bP_x\{\tau^+_\infty<\infty\}>0$ for all large $x$.
\end{itemize}
\end{thm}

The proof of Theorem \ref{t3.1} is deferred to Section \ref{proofs}.

\bcondition\label{c3}
For any $x$ and $b$ with $b>x>0$,
 \beqlb\label{ca}
\bP_x\{\tau^+_b<\infty \}>0.
 \eeqlb
\econdition

Putting Theorem \ref{t3.1} (ii) and the above  condition together we reach the following remark.
\bremark
If Condition \ref{c3} and the assumption in Theorem \ref{t3.1} (ii) hold,
then $\bP_x\{\tau^+_\infty<\infty \}>0$ for all $x>0$.
\eremark

The proof for the next result is deferred to the end of  Section \ref{Foster_Lyapunov}.
\bproposition\label{L0.2}
Given $b>x>0$, if there exists $a\in (0,x)$ so that
 \beqnn
\inf_{y\in[a,b]}\gamma_1(y)+\inf_{y\in[a,b]}\gamma_2(y)> 0,
 \eeqnn
then \eqref{ca} holds.
\eproposition

\subsection{ Coming down from infinity}

We say that the process $X$ comes down from infinity if
\begin{equation}\label{comedown_def}
\lim_{b\goto\infty}\lim_{x\goto\infty}\bP_x \{\tau^-_b<t\}=1 \quad\text{ for all}\quad t>0,
\end{equation}
 and it stays infinite if
\begin{equation*}
\lim_{x\goto\infty}\bP_x \{\tau^-_b<\infty\}=0\quad\text{ for all}\quad b>0.
\end{equation*}


 We first present equivalent conditions for coming down from infinity. From the proof one can see that they hold for any real-valued Markov processes with no downward jumps.

\begin{prop}
The following statements are equivalent:
\begin{itemize}
\item[(i)]
Process $X$ comes down from infinity.
\item[(ii)]
$\lim_{x\goto\infty}\mbb{E}_x[\tau^-_b]<\infty $ for all large $b$.
\item[(iii)]
\begin{equation}\label{comedown_equ}
\lim_{b\goto\infty}\lim_{x\goto\infty}\mbb{E}_x[\tau^-_b]=0.
\end{equation}
\end{itemize}
\end{prop}

\begin{proof}
For the proof that (i) implies (ii), we refer to the proof of Theorem 1.11 of Li (2018).

Suppose that (ii) holds. Then for any $x'>b$, we have
 \begin{equation}\label{equivalent_a}
 \lim_{x\goto\infty}\mbb{E}_x [ \tau^-_{b}]=\lim_{x\goto\infty}(\mbb{E}_x[\tau^-_{x'}]+ \mbb{E}_{x'}[\tau^-_{b}]).
 \end{equation}
 First letting $x'\goto\infty$, and then letting $b\goto\infty$ in (\ref{equivalent_a}), we obtain (\ref{comedown_equ}). (iii) thus holds.


 (i) follows from (iii) by the Markov inequality.

\end{proof}

\begin{thm}\label{comingdown}
\begin{itemize}
\item[{\normalfont(i)}]
If there exist constants $a>1$ and $r<1$ such that
\begin{equation}\label{condition_stay}
G_a(u)\geq -(\ln u)^r
\end{equation}
 for  all $u$ large, then process $X$ stays infinite.
\item[{\normalfont(ii)}]
If there  exist constants $0<a<1, r>1$ such that
\begin{equation}\label{comingdown_a}
G_a(u)\geq (\ln u)^r
\end{equation}
 for all $u$ large enough, then process $X$ comes  down  from infinity.
\end{itemize}
\end{thm}

The proof of Theorem \ref{comingdown} is deferred to Section \ref{proofs}.

\begin{remark}
More recently, for the process $X$ with $\gamma_0=\gamma_1=\gamma_2$, the speeds of coming down from infinity are studied in details in   Dawson et al. (2018) for the cases that either the function $\gamma_i$ is regularly varying at infinity or $\gamma_i(x)=g(x)\e^{\theta x}$ for $\theta>0$ and  function $g$ that is regularly varying at infinity.
\end{remark}

\subsection{An application: weighted total population}\label{population}


Let  $\gamma$ be a strictly positive function defined on $[0,\infty)$ that is bounded
on any bounded interval.
In the following,
we consider the  weighted occupation time, or the weighted total population of $X$ before explosion, defined as
 \beqnn
S=\int_0^{\tau_0^-\wedge \tau_\infty^+}\gamma(X_s)\dd s.
 \eeqnn
For $t\ge0$ define
 \beqnn
U_t:=\int_0^{t\wedge \tau_0^-\wedge \tau_\infty^+}\gamma(X_s)\dd s
~~\text{and}~~
V_t:=\inf\{s> 0:U_s> t\}.
 \eeqnn
Define the process $\bar{X}\equiv\{\bar{X}_t:t\ge0\}$ by $\bar{X}_t:=X_{V_t}$ for $V_t<\infty$
and $\bar{X}_t:=X_\infty:=\limsup_{t\to\infty}X_t$ for $V_t=\infty$.
Define  stopping times $\bar{\tau}_0^-$ and $\bar{\tau}_\infty^+$
similarly to $\tau_0^-$ and $\tau_\infty^+$, respectively, with $X$ replaced by $\bar{X}$.

   We first observe that with the above-mentioned Lamperti type transform, a time changed solution to the generalized  Dawson-Li equation (\ref{sdeB})  remains a solution to another  generalized  Dawson-Li equation.

We leave the proof of the next result to the interested readers.

\btheorem\label{t1.13}
For $i=0,1,2$ and $y>0$ define $\bar{\gamma}_i(y):=\gamma_i(y)/\gamma(y)$.
Then there exist, on an extended probability space, a Gaussian white noise
$\{W_0(\dd s, \dd u) : s \ge0, u>0\}$
with intensity $\dd s \dd u$
and an independent compensated Poisson random measure
$\{\tilde{N}_0(\dd s, \dd z, \dd u) : s\ge0, z>0, u > 0\}$
with intensity $\dd s\pi(\dd z)\dd u$
so that $\{\bar{X}_t: t\ge 0\}$ solves the following SDE:
 \beqlb\label{1.5}
\bar{X}_t
 \ar=\ar
x+\int_0^t\bar{\gamma}_0(\bar{X}_s)\dd s
+\int_0^t\int_0^{\bar{\gamma}_1(\bar{X}_s)}W_0(\dd s,\dd u)\cr
\ar\ar\qquad+\int_0^t\int_0^\infty\int_0^{\bar{\gamma}_2(\bar{X}_{s-})}
z\tilde{N}_0(\dd s,\dd z,\dd u)
 \eeqlb
for $0\le t<\bar{\tau}_0^-\wedge\bar{\tau}_\infty^+$.
\etheorem

 We leave the proof of the next key observation to interested readers.

\begin{prop}\label{t1.2} We have
$S =\bar{\tau}_0^-\wedge\bar{\tau}_\infty^+$.
\end{prop}

\bremark\label{rem_population}
By Proposition \ref{t1.2} and  Theorem \ref{t1.13}, the finiteness for $S$ is translated into
 extinction and explosion behaviors for the time changed process $\bar{X}$ for which we can apply Theorems \ref{t2.5} and \ref{t3.1}. More details can be found later in Example \ref{example} in Section \ref{subsec_power}.
 If $\gamma(x)=\gamma_1(x)=\gamma_2(x)\equiv x$ and $\gamma_0$ satisfies certain interaction condition,    
then the behaviors for $S$ have been studied in Theorems 4.3.1 and 4.3.2 of Le (2014).
\eremark

\subsection{Processes with power  branching rate functions}\label{subsec_power}

To obtain more explicit results, in this subsection we only consider processes with power function branching rates, i.e.
 \[\gamma_i(x)=b_i x^{r_i},\qquad x>0, \,\,\, i=0,1,2, \]
for $r_0, r_1, r_2\geq 0$, $b_0\in \bR$,  $b_1, b_2\geq 0, b_1+b_2>0 $. In addition, we assume that the measure $\pi$ is defined in \eqref{pi} with $1<\alpha<2$. 
Then  by \eqref{H_a},
\begin{equation*}
\begin{split}
G_a(u)=&(a-1)u^{-1} b_0u^{r_0}-2^{-1}a(a-1)u^{-2} b_1u^{r_1}\\
&-  a(a-1)u^{-\alpha} b_2u^{r_2}c_{\alpha,a},
\end{split}
\end{equation*}
 where
\[c_{\alpha, a}:=\frac{\alpha(\alpha-1)}{\Gamma(2-\alpha)}\int_0^\infty y^{1-\alpha}\dd y\int_0^1 (1+yv)^{-1-a}(1-v)\dd v.\]
It is easy to see from properties of the beta function that
 \beqnn
c_{\alpha,1}
 \ar:=\ar
\frac{\alpha(\alpha-1)}{\Gamma(2-\alpha)}\int_0^\infty y^{1-\alpha}\dd y\int_0^1 (1+yv)^{-2}(1-v)\dd v \cr
 \ar=\ar
\frac{\alpha(\alpha-1)}{\Gamma(2-\alpha)}\int_0^\infty x^{1-\alpha}(1+x)^{-2}\dd x\int_0^1v^{\alpha-2}(1-v)\dd v \cr
 \ar=\ar
\frac{\alpha(\alpha-1)}{\Gamma(2-\alpha)}
\times
\frac{\Gamma(2-\alpha)\Gamma(\alpha)}{\Gamma(2)}
\times
\frac{\Gamma(\alpha-1)\Gamma(2)}{\Gamma(\alpha+1)}
=\Gamma(\alpha).
 \eeqnn

\bexample\label{r1.1}
In order to apply Theorems \ref{t2.5}, \ref{t3.1} and \ref{comingdown}, we only need to compare powers and coefficients of the three terms in the polynomial $G_a(u)$ for $0<a<1$ or $a>1 $, respectively.
To handle the critical case of $r_1=r_0+1$ or (and) $r_2=r_0+\alpha-1$ where some terms have the same power, we further choose the value of $a$ close enough to $1$ to obtain the best possible results. For instance, if both $r_1=r_0+1$ and $r_2=r_0+\alpha-1$ hold,
for $b_0>b_1/2+c_{\alpha,1} b_2$, we  choose the constant $a$ satisfying $1<a<b_0/(b_1/2+c_{\alpha, a} b_2)$, and for
$b_0<b_1/2+c_{\alpha,1} b_2$, we  choose the constant $a$ satisfying $(b_0/(b_1/2+c_{\alpha,a} b_2))\vee 0<a<1$.
\eexample

By Theorem \ref{t2.5} and Proposition \ref{thmA}, we can obtain explicit and very sharp conditions of extinction/non-extinction for the process $X$ in Example \ref{r1.1}.

For non-extinction we have   $\bP_x\{\tau^-_0=\infty\}=1 $ for all $x>0$ if one of the following
two sets of conditions holds.
\begin{itemize}
\item[(i)]
$b_0\leq 0$ and all of the following hold.
\begin{itemize}
\item[(ia)]
if $b_0<0$,  then $r_0\geq 1$;
\item[(ib)]
if $b_1>0$, then $r_1\geq 2$;
\item[(ic)]
if $b_2> 0$, then $r_2\geq\alpha$;
\end{itemize}
\item[(ii)]
$b_0>0$ and all of the following hold.
\begin{itemize}
\item[(iia)]
if $b_1>0$, then $r_1\geq (r_0+1)\wedge 2$; 
\item[(iib)]
if $b_2>0$, then  $r_2\geq  (r_0-1+\alpha)\wedge\alpha$;  
\item[(iic)]
 \[b_0 > \frac{b_1}{2} 1_{\{r_1=r_0+1<2\}}+\Gamma(\alpha) b_2 1_{\{r_2=r_0+\alpha-1<\alpha\}}.\]
\end{itemize}
\end{itemize}
In addition, under condition (i), for all $x>0$
$$\bP_x\{\tau^-_0=\infty \,\,\, \text{and}\,\,\, X_t\goto 0 \,\, \text{ as} \,\, t\goto\infty\}=1,$$
i.e., extinguishing occurs.

For extinction with a positive probability we have $\bP_x\{\tau^-_0<\infty\}>0$ for all $x>0$ if one of the following two sets of conditions holds.
\begin{itemize}
\item[(i)]
$b_0\leq 0$ and at least one of the following hold.
\begin{itemize}
\item[(ia)]
$b_0<0$ and $r_0<1$;
\item[(ib)]
$b_1> 0$ and $r_1<2$;
\item[(ic)]
$b_2> 0$ and $r_2<\alpha$.
\end{itemize}
\item[(ii)] $b_0> 0$ and at least one of the following hold.
\begin{itemize}
\item[(iia)]
 $b_1>0$ and $r_1<(r_0+1)\wedge 2$;
\item[(iib)]
 $b_2> 0$ and $r_2<(r_0+\alpha-1)\wedge \alpha$;
\item[(iic)]
 \begin{equation*}
 \begin{split}
 b_0<
 &\frac{b_1}{2} 1_{\{r_1=r_0+1<2\}}+\Gamma(\alpha) b_2 1_{\{r_2=r_0+\alpha-1<\alpha\}}.\\
 \end{split}
 \end{equation*}
\end{itemize}
\end{itemize}
In addition, $\bP_x\{\tau^-_0<\infty\}=1 $ for all $x>0$  under condition (i).

\begin{remark}
Note that the above Condition (iic) for $\bP_x\{\tau^-_0=\infty\}=1 $ and the above Condition (iic) for $\bP_x\{\tau^-_0<\infty\}>0$ agree with the corresponding results in Berestycki {\it et al.} (2015); see the corresponding comments in Section \ref{nonlinear}.
\end{remark}

By Theorem \ref{t3.1} and Proposition \ref{L0.2}, we obtain rather sharp conditions of explosion/non-explosion for the process $X$ in Example \ref{r1.1}.

For non-explosion we have $\bP_x\{\tau^+_\infty<\infty\}=0$
for all $x>0$  if either $b_0\leq 0$ or at least one of the following is true.
\begin{itemize}
\item[(i)]
$b_0>0$ and $r_0\leq 1$.
\item[(ii)]
$b_0>0, \, r_0> 1 $ and at least one of the following hold.
\begin{itemize}
\item[(iia)]
 $b_1>0$ and $r_1>r_0+1$;
\item[(iib)]
 $b_2> 0$ and $r_2>r_0+\alpha-1$;
\item[(iic)]
\[b_0<\frac{b_1}{2} 1_{\{r_1=r_0+1\}}+\Gamma(\alpha) b_2 1_{\{r_2=r_0+\alpha-1\}}. \]
\end{itemize}
\end{itemize}

For explosion with a positive probability we have
$\bP_x\{\tau^+_\infty<\infty\}>0$
for all $x>0$   if
$b_0>0, r_0>1$ and all of the following hold.
\begin{itemize}
\item[(i)]
if $b_1>0$, then $r_1\leq r_0+1$;    
\item[(ii)]
if $b_2> 0$, then $r_2\leq r_0+\alpha-1$;                 
\item[(iii)]
 \[b_0>\frac{b_1}{2} 1_{\{r_1=r_0+1\}}+\Gamma(\alpha) b_2 1_{\{r_2=r_0+\alpha-1\}}.\]
\end{itemize}

Similarly, by Theorem \ref{comingdown} we obtain rather sharp conditions for coming down from infinity.

The process $X$ in Example \ref{r1.1} comes down from infinity if one of the following holds
\begin{itemize}
\item[(i)]
$b_0\leq 0$ and at least one of the following hold. 
\begin{itemize}
\item[(ia)]
 $b_0<0$ and $r_0>1$;
\item[(ib)]
 $b_1>0$ and $r_1>2$ ;
\item[(ic)]
 $b_2>0$ and $r_2>\alpha$.
\end{itemize}
\item[(ii)]
 $b_0> 0$ and at least one of the following hold.
\begin{itemize}
\item[(iia)]
 $b_1>0$ and $r_1>(r_0+1)\vee2$;
\item[(iib)]
 $b_2> 0$ and $r_2>(r_0+\alpha-1)\vee\alpha$;
\item[(iic)]
 \begin{equation*}
 \begin{split}
 b_0<
 &\frac{b_1}{2} 1_{\{r_1=r_0+1>2\}}+\Gamma(\alpha) b_2 1_{\{r_2=r_0+\alpha-1>\alpha\}}.\\
 \end{split}
 \end{equation*}
\end{itemize}
\end{itemize}

Process $X$ in Example \ref{r1.1} stays infinite if at least one of the following hold.
\begin{itemize}
\item[(i)]
 $b_0\leq 0$ and all of the following hold.
\begin{itemize}
\item[(ia)]
if $b_0<0$, then $r_0\le1$;
\item[(ib)]
if  $b_1>0$, then $r_1\le2$;
\item[(ic)]
if $b_2>0$, then  $r_2\le\alpha$.
\end{itemize}
\item[(ii)]
 $b_0> 0$ and all of the following hold.
\begin{itemize}
\item[(iia)]
if $b_1>0$, then $r_1\le(r_0+1)\vee2$;
\item[(iib)]
if $b_2> 0$, then  $r_2\le(r_0+\alpha-1)\vee\alpha$;
\item[(iic)]
 \begin{equation*}
 \begin{split}
 b_0>
 &\frac{b_1}{2} 1_{\{r_1=r_0+1>2\}}+\Gamma(\alpha) b_2 1_{\{r_2=r_0+\alpha-1>\alpha\}}.\\
 \end{split}
 \end{equation*}
\end{itemize}
\end{itemize}

From the above example we make the following observations.

\begin{remark}
\begin{itemize}
\item[{\normalfont(i)}]
There is no extinction if the process $X$ has a small enough negative drift together with small enough fluctuations near $0$. If $X$ has a positive drift, then the requirements on the fluctuations are weaker.
Extinction happens with a positive probability if $X$ has either a large enough negative drift or large enough fluctuations near $0$. Even if $X$ has a small positive drift near $0$, extinction can still happen with a positive probability if the fluctuations are large enough.
\item[{\normalfont(ii)}]
The explosion is caused by a large enough drift  associated with the function $\gamma_0$. The fluctuations of the process $X$ associated with the functions $\gamma_1$ and $\gamma_2$ cannot cause explosion. But large enough fluctuations can prevent the explosion from happening.
\item[{\normalfont(iii)}]
A large enough negative drift or large enough fluctuations near infinity can cause coming down from infinity. Even if the process $X$ has a positive drift, large enough fluctuations can still cause coming down from infinity.
On the other hand, the process $X$ with a moderate negative drift and moderate fluctuations near infinity stays infinite, and  if it allows large fluctuations, with a large enough positive drift it can still stay infinite.
\end{itemize}
\end{remark}

\begin{remark}
If $b_2=0$, then $X$ is a diffusion whose explosion behavior is characterized by Feller's criterion; see e.g. Corollary 4.4 of Cherny and Engelbert  (2005). One can check that the explosion/nonexplosion conditions in Example \ref{r1.1} are consistent with it.
\end{remark}

\begin{remark}
Example \ref{r1.1}  recovers,  for the case with spectrally positive stable L\'evy measure specified in (\ref{pi}), the integral tests for extinction, explosion and coming down from infinity in Theorems 1.7, 1.9 and 1.11 of Li (2018),
which were proved using a very different approach.
Recall that the continuous-state polynomial branching process in Li (2018) is the process $X$ with power branching rate functions satisfying
$r_i=r, i=0,1,2$.
By Example \ref{r1.1} we have for the continuous-state polynomial branching process,
\begin{itemize}
\item
 $\bP_x\{\tau^-_0<\infty\}>0$ for all $x>0$, i.e. extinction occurs,  if and only if
\[21_{\{b_1\ne 0\}}+\alpha 1_{\{b_1=0, b_2\ne 0\}}>r;\]
\item $\bP_x\{\tau^+_\infty<\infty\}>0$
for all $x>0$, i.e. explosion occurs,   if and only if $b_0>0$ and $r>1$;
\item
The process $X$ comes down from infinity  if and only if $ b_0\le 0$ and
\[1_{\{b_0\neq 0\}}+\alpha1_{\{b_0=0, b_2\ne0\}}+21_{\{b_0=0, b_1\ne0, b_2=0\}}< r;\]
\end{itemize}
which agree the  integral tests in Li (2018).
\end{remark}

The next example is on the finiteness of the weighted total population $S$ of $X$ introduced in Section \ref{population}. The next results follow from Remark \ref{rem_population}, Theorem \ref{t1.13} and Example \ref{r1.1}.

\bexample\label{example}
 Let $\gamma(x)=x^r$ for $0<r< \min\{r_0, r_1, r_2\}$ in Theorem \ref{t1.13}.
Observe that $\bP_x\{S=\bar{\tau}_0^-\wedge\bar{\tau}^+_\infty=\infty\}=1$ if and only if $\bP_x\{\bar{\tau}^-_0=\infty\}=1 $
and $\bP_x\{\bar{\tau}^+_\infty=\infty\}=1 $.
The conditions for
$\bP_x\{S<\infty\}=\bP_x\{\bar{\tau}^-_0\wedge\bar{\tau}^+_\infty<\infty\}=0$  for  $x>0$ can be found in Example \ref{r1.1}.


Similarly,
observe that $\bP_x\{\bar{\tau}^-_0\wedge\bar{\tau}^+_\infty<\infty\}>0$ if and only if $\bP_x\{\bar{\tau}^-_0<\infty\}>0 $
or $\bP_x\{\bar{\tau}^+_\infty<\infty\}>0 $.
Then the conditions for $\bP_x\{S<\infty\}>0$  can also be found in Example \ref{r1.1}.
\eexample

\section{Existence and uniqueness of solution}\label{main}

\setcounter{equation}{0}

In this section we find conditions on the functions $\gamma_i, i=0, 1,2$ under which SDE {\rm(\ref{sdeB})} has a pathwise unique solution $X$, and  consequently $X$  is a Markov process.  For this purpose, we only need the functions $\gamma_i, i=0, 1,2$ to be locally Lipschitz because we only consider the solution up to the first time of hitting $0$ or explosion.


\btheorem\label{thm00}
Suppose that the functions $\gamma_i,$ $i=0,1,2$ are locally Lipschitz; i.e.,
for each closed interval $A\subset (0,\infty)$, there is a constant $c(A)>0$ so that for any $x,y\in A$,
 \beqnn
|\gamma_0(x)-\gamma_0(y)|+|\gamma_1(x)-\gamma_1(y)|+ |\gamma_2(x)-\gamma_2(y)|\le c(A)|x-y|.
 \eeqnn
 Then
\begin{itemize}
\item[{\normalfont(i)}]
 For any initial value $X_0=x\ge0$, there exists
a pathwise unique solution (defined at the beginning of Section \ref{main_result}) to SDE {\rm(\ref{sdeB})}.
\item[{\normalfont(ii)}]
If in addition, $\gamma_2$ is an increasing function, then for any $y\ge x\in [0,\infty)$ and solutions $X^x:= (X_t^x)_{t\ge 0}$ and $X^y:= (X_t^y)_{t\ge 0}$ to SDE {\rm(\ref{sdeB})} with $X^x_0=x$ and $X^y_0=y$, we have
$$\bP\{X_t^y\ge X_t^x \,\,\, \text{ for all} \,\,\, t\ge 0\}=1.$$
 \end{itemize}
\etheorem
\proof  
(i) We prove the result by an approximation argument. For each $n\ge 1$ and $i=0,1,2$ define
\beqnn
\gamma_i^n(x):=\left\{
\begin{array}{lcl}
 \gamma_i(n), & & n< x< \infty, \cr
 \gamma_i(x), & & 1/n\le x\leq n, \cr
 \gamma_i(1/n), & & 0\le x<  1/n.
\end{array} \right.
 \eeqnn
 By pp.245--246 in Ikeda and Watanabe (1989),  for each $n\ge1$ there is a unique strong solution $(\xi^n_t)_{t\ge 0}$ to
 \beqlb\label{e0.2}
\xi^n_t
\ar=\ar
x+\int_0^t\gamma_0^n(\xi^n_s)\dd s
+\int_0^t\int_0^{\gamma_1^n(\xi^n_s)}
W(\dd s,\dd u)\cr
\ar\ar\qquad+\int_0^t\int_{0}^\infty\int_0^{\gamma_2^n(\xi^n_{s-})}
z\tilde{N}(\dd s, \dd z, \dd u).
 \eeqlb
For $m,n\ge1$ define stopping time
$$\tau^n_m:= \inf\{t\ge 0: \xi^n_t\geq m \text{\,\, or \,\,}\xi^n_t\leq 1/m\}.$$
 Then we have $\xi^n_t = \xi^m_t$ for $t\in [0,\tau_{m\wedge n}^{m\wedge n})$ and $\tau^{n+i}_n=\tau^n_n$, $i=1,2,\ldots$ Clearly, the sequence of stopping times $\{\tau_n^n\}$ is increasing in $n$.
Let $\tau := \lim_{n\to \infty} \tau^n_n$.
We define the process $X:=(X_t)_{t\ge0}$ by $X_t = \xi^n_t$ for $t\in [0,\tau^n_n)$ and $X_t = \limsup_{n\to \infty}\xi^n_{\tau^n_n}$ for $t\in [\tau,\infty)$. Then
$$\tau^n_n:= \inf\{t\ge 0: X_t\geq n \text{\,\, or\,\,} X_t\leq 1/n\}$$
 and $X$ is a solution of (\ref{sdeB}). Since the pathwise uniqueness of the solution holds for (\ref{e0.2}) in the time interval $[0,\tau^n_n)$ for each $n\ge 1$, there exists a pathwise unique solution to {\rm(\ref{sdeB})}.

(ii) Let $(\xi_n^x(t))_{t\geq 0}$ denote the solution
of (\ref{e0.2}) to indicate its dependence on the initial state. To apply Theorem 2.2 in Dawson and Li (2012), we identify the notation in Dawson and Li (2012) with that in this paper in the following equations, where the notation on the left hand sides comes from Dawson and Li (2012) and that on the right hand sides is from the present paper. 
 \begin{equation*}
 \begin{split}
&E=(0,\infty),~ U_0=(0,\infty)^2, ~\pi(\dd u)=\dd u, ~  g_1(x,z,u)\equiv0, \\
&
\mu_0(\dd z,\dd u)=\pi(\dd z)\dd u, ~
\tilde{N}_0(\dd s, \dd z,\dd u)=\tilde{N}(\dd s, \dd z,\dd u)
 \end{split}
 \end{equation*}
and
 \beqnn
b(x)=b_1(x)=\gamma_0^n(x),~\sigma(x,u)=1_{\{u\le \gamma_1^n(x)\}},~
g_0(x, z, u)=z1_{\{u\le \gamma_2^n(x)\}}.
 \eeqnn
Then conditions (2.a,b,c) in Dawson and Li (2012) are satisfied due to the Lipschitz
properties of $\gamma_i^n$ for $i=0,1,2$.
Let
$$l_0(x,y,u):=1_{\{u\le \gamma_2(x)\}}-1_{\{u\le \gamma_2(y)\}}.$$
Since the function $\gamma_2(x)$ is non-decreasing in $x$, then  for $x<y$ we have
 \beqnn
I(x,y)
 \ar:=\ar
\int_0^\infty\dd u \int_0^1\frac{l_0(x,y,u)^2(1-t)1_{\{|l_0(x,y,u)|\le n\}}}{|(x-y)+tl_0(x,y,u)|}\dd t \cr
 \ar=\ar
\int_{\gamma_2(x)}^{\gamma_2(y)}
\dd u \int_0^1\frac{1-t}{|(x-y)-t|}\dd t \cr
 \ar\le\ar
(\gamma_2(y)-\gamma_2(x))(\ln(y-x+1)-\ln(y-x))<\infty.
 \eeqnn
Similarly, $I(x,y)<\infty$ for all $x\ge y$.
Then
condition (2.d) of Theorem 2.2 in Dawson and Li (2012) holds.
Now for any $y\ge x\ge 0$, by  Theorem 2.2 in Dawson and Li (2012) we can show that $\xi_n^y(t)\ge \xi_n^x(t)$ a.s. for all $n$ and $t\ge 0$. Consequently,  $X_t^y\ge X_t^x$ a.s. for all  $t\ge 0$.\qed




 Throughout the rest of this paper, we always assume that SDE {\rm(\ref{sdeB})} has a unique weak solution which is a Markov process.

\begin{remark}
The solution to SDE {\rm(\ref{sdeB})} also arises as the weak limit in the Skorokhod space $D([0, \infty), \mathbb{R}_+) $ for a sequence of discrete-state and continuous-time Markov chains that can be interpreted as discrete-state branching processes with population dependent branching rates; see Li et al. (2018) for more details.
\end{remark}

\section{Foster-Lyapunov criteria for extinction and explosion}\label{Foster_Lyapunov}

\setcounter{equation}{0}

In this section, we first present Foster-Lyapunov criteria type results for the process $X$ which generalize a similar result for Markov chains; see Chen (2004, p.84).

Let $C^2[0,\infty)$ be the space of twice continuously
differentiable functions on $[0,\infty)$.
Define the operator $L$ on $C^2[0,\infty)$ by
 \beqnn
Lg(y)
 \ar:=\ar
\gamma_0(y)g'(y)+\frac{1}{2}\gamma_1(y)g''(y) \cr
 \ar\ar\quad
+\gamma_2(y)\int_0^\infty(g(y+z)-g(y)-zg'(y))\pi(\dd z).
 \eeqnn

\blemma\label{lemA}
 Given $a\geq 0$,
let $g\in C^2[0,\infty)$ be a non-negative function  satisfying  the following
conditions:
\begin{itemize}
\item[{\normalfont(i)}] $\sup_{y\in [a, b)} | Lg(y)|< \infty$ for all $b>a$,  i.e.,  $Lg$ is locally bounded on $[a, \infty) $;
\item[{\normalfont(ii)}] $\sup_{y\in [a, \infty)} g(y)< \infty$;
\item[{\normalfont(iii)}] $g(a)>0$ and $\lim_{y\goto\infty} g(y)=0$;
\item[{\normalfont(iv)}] For all $b>a$, there is a constant $d_b>0$ so that $Lg(y)\geq d_b g(y)$ for all  $y\in(a,b)$.
\end{itemize}
Then for any $x>a$, we have
 \beqlb\label{con0}
\bP_x\{\tau^-_a<
\infty\}\ge g(x)/g(a).
 \eeqlb
\elemma
\proof
For any $b>x>a$, by It\^o's formula and conditions (i) and (ii), we have
\beqnn  g(X_{t\wedge\tau^-_a\wedge\tau^+_b})=g(x)+\int_0^{t\wedge\tau^-_a\wedge\tau^+_b }Lg(X_s)\dd s+\text{mart}.\eeqnn
Taking expectations on both sides, we have
\beqnn  \bE_x \big[g(X_{t\wedge\tau^-_a\wedge\tau^+_b})\big]=g(x)+\int_0^t\bE_x  \Big[Lg(X_s)1_{\{s< \tau^-_a\wedge\tau^+_b\}}\Big]\dd s.\eeqnn
By integration by parts,
 \beqnn
 \ar\ar
\int_0^\infty \e^{-d_bt}
\bE_x \Big[Lg(X_t)1_{\{t<\tau^-_a\wedge\tau^+_b\}}\Big]\dd t \cr
 \ar\ar\quad=
\int_0^\infty \e^{-d_bt}\dd \bE_x \Big[g(X_{t\wedge\tau^-_a\wedge\tau^+_b})\Big]\cr
 \ar\ar\quad
=d_b\int_0^\infty \e^{-d_bt} \bE_x
\Big[g(X_{t\wedge\tau^-_a\wedge\tau^+_b})\Big]\dd t-g(x).
 \eeqnn
Then by (iv),
 \beqnn
 \ar\ar
d_b\int_0^\infty \e^{-d_bt} \bE_x
\Big[g(X_{t\wedge\tau^-_a\wedge\tau^+_b})\Big]\dd t-g(x) \cr
 \ar\ar\qquad
\geq d_b\int_0^\infty \e^{-d_bt}\bE_x \Big[g(X_t)1_{\{t<
\tau^-_a\wedge\tau^+_b\}}\Big]\dd t. \eeqnn
It follows that
 \beqnn
g(x)
 \ar\leq \ar
d_b\int_0^\infty \e^{-d_bt}\bE_x \Big[g(X_{\tau^-_a\wedge\tau^+_b})1_{\{t\ge\tau^-_a\wedge\tau^+_b\}}\Big]\dd t \cr
 \ar\leq\ar
g(a)\bP_x\{\tau^-_a<\infty\}+\sup_{y\ge b}g(y).
 \eeqnn
Inequality (\ref{con0}) thus follows by letting $b\goto\infty$ and (iii).
\qed

 The proof for the next lemma is similar to that of Lemma \ref{lemA} and we omit it.

\blemma\label{lem1}
Given $0< x< b$, suppose there exist constants $a\in[0,x)$, $d>0$ and
a function $g\in C^2[0,\infty)$  satisfying  the following
conditions:
\begin{itemize}
\item[{\normalfont(i)}] $\sup_{y\in [a, b]} | Lg(y)|< \infty$;
\item[{\normalfont(ii)}] $\sup_{y\in [a, b)} |g(y)|< \infty$;
\item[{\normalfont(iii)}] $g(a)=0$ and $g(x)>0$;
\item[{\normalfont(iv)}] $Lg(y)\geq d g(y)$ for all  $y\in[a,b]$.
\end{itemize}
Then we have
$\bP_x\{\tau^+_b<\infty\}>0$.
\elemma


 As applications of Lemmas \ref{lemA} and \ref{lem1}, we prove Propositions \ref{thmA} and \ref{L0.2} in this section.


\noindent{\it Proof of Proposition \ref{thmA}.} 
(i) Let $g(y)=\e^{-\lambda y}$ with
$\lambda>0$ large enough.  Note that $g$ satisfies the conditions of
Lemma \ref{lemA}. Then by Lemma \ref{t4.1}, we have uniformly for all $a<y<b$,
\beqlb\label{1.1}
Lg(y)
 \ar\geq\ar
\lambda \e^{-\lambda y}\Big\{-\sup_{ a\le z\le b}(\gamma_0(z)\vee 0)+\frac{\lambda}{2}\inf_{a\le z\le b} \gamma_1(z) \cr
 \ar\ar\quad
+\lambda\inf_{a\le z\le b}\gamma_2(z)\int_0^\infty z^2\pi(\dd z)\int_0^1 \e^{-\lambda zu}(1-u)\dd u\Big\}.
 \eeqlb
Observe that
 \beqnn
 \ar\ar
\la\int_0^\infty z^2\pi(\dd z)\int_0^1 \e^{-\lambda zu}(1-u)\dd u
 \ge
2^{-1}\la\int_0^\infty z^2\pi(\dd z)\int_0^{1/2} \e^{-\lambda zu}\dd u \cr
 \ar\ar\quad\ge
2^{-1}\int_0^\infty z(1- \e^{-\la z/2})\pi(\dd z)
\ge
2^{-1}(1- \e^{-1/2})\int_{1/\la}^\infty z\pi(\dd z)
 \eeqnn
converges   to $\int_0^\infty z\pi(\dd z)=\infty$ as $\lambda\to\infty$.
It then follows that for each $b>a$ there is a constant $d_b(\lambda)>0$ so that
 \beqlb\label{1.2}
Lg(y)\ge d_b(\lambda)\e^{-\la y},\qquad a<y<b
 \eeqlb
as $\la$ large enough.
Thus by Lemma \ref{lemA}, for $x>a$ and large enough $\la$,
\begin{equation}\label{thmAa}
\bP_x\{\tau^-_a<\infty\}\geq \e^{-\lambda (x-a)}>0,
\end{equation}
which gives \eqref{Condition_a}.

(ii) Suppose that there is a constant $c>0$ so that $\gamma_0(y)\le 0$ for
all $y\ge c$. Similar to the argument in \eqref{1.1} and \eqref{1.2},
given any $\lambda>0$, uniformly for $c\vee a<y<b$, we have
\beqnn
Lg(y)
 \ar\geq\ar
\frac{\lambda^2}{2}\gamma_1(y)\e^{-\lambda
y}+\lambda^2\gamma_2(y)\e^{-\lambda y}\int_0^\infty z^2\pi(\dd z)\int_0^1
\e^{-\lambda zu}(1-u)\dd u \cr
 \ar\geq\ar
d_b'(\lambda) \e^{-\lambda y}
\eeqnn
for some constant $d_b'(\la)>0$.
It follows again from Lemma
\ref{lemA} that  for all $\lambda>0$,
\beqnn  \bP_x\{\tau^-_l<\infty\}\geq \e^{-\lambda (x-l)}>0, \qquad x>l\ge c\vee a.\eeqnn
Letting $\lambda\goto 0$,  we have
\begin{equation}\label{thmAb}
\bP_x\{\tau^-_l<\infty\}=1,\qquad x>l\ge c\vee a.
\end{equation}
It follows from  \eqref{thmAa} that for large enough $\la$,
 \beqlb\label{1.3}
\bP_x\{\tau^-_a<\infty\}\geq \e^{-\lambda (x-a)}, \qquad x>a.
 \eeqlb

For any $x>a>0$ and $t>0$, combining (\ref{thmAb}) and \eqref{1.3}, by the strong Markov property, we have
 \beqlb\label{thmAc}
 \ar\ar
\bP_x\{\tau^-_a<\infty\} \cr
 \ar=\ar
\bP_x\{\tau^-_a<t\}+\int_a^{c\vee a}
\bP_x\big\{t\leq\tau^-_a<\infty, \, X_t\in \dd z\big\}
\bP_z\{\tau^-_a<\infty \} \cr
 \ar\ar
+\int_{c\vee a}^\infty
\bP_x\big\{t\leq\tau^-_a<\infty, \, X_t\in \dd z\big\}
\bP_z\{\tau^-_a<\infty \} \cr
 \ar\ge\ar
\bP_x\{\tau^-_a<t\}+\int_a^{c\vee a}
\bP_x\big\{t\leq\tau^-_a<\infty, \, X_t\in \dd z\big\}
\bP_{c\vee a}\{\tau^-_a<\infty \} \cr
 \ar\ar
+\int_{c\vee a}^\infty
\bP_x\big\{t\leq\tau^-_a<\infty, \, X_t\in \dd z\big\}
\bP_{c\vee a}\{\tau^-_a<\infty \} \cr
 \ar\geq\ar
\bP_x\{\tau^-_a<t\}+\e^{-\lambda(c\vee a-a)}(1-\bP_x\{\tau^-_a<t\}).
 \eeqlb
Letting $t\goto\infty$ in (\ref{thmAc}), we have
$  \bP_x\{\tau^-_a<\infty\}=1$.
The desired result then follows.

(iii)
For any small enough $\ep>0$, let
 \beqnn
A_n:=\Big\{\tau^-(\ep^{n+1})<\infty,
\tau^+_\ep\circ\theta_{\tau^-(\ep^{n+1})} <\tau^-{(\ep^{n+2})}\circ\theta_{\tau^-(\ep^{n+1})} \Big\},~ n\ge1.
 \eeqnn
Since $\gamma_0(y)\le0$ for all $y\in\mbb{R}$, then
$(X_t)_{t\ge0}$ is a supermartingale, which implies
 \beqnn
\ep^{n+1}=X_{\tau^-(\ep^{n+1})}
\ge\mbb{E}_{\ep^{n+1}}\Big[X_{\tau^+_\ep\wedge\tau^-{(\ep^{n+2})}}\Big]
\ge\ep\mbb{P}_{\ep^{n+1}}\{\tau^+_\ep< \tau^-_{\ep^{n+2}}\}
 \eeqnn
by optional stopping.
Thus,
 \beqnn
\bP_x\{A_n\}\le \mbb{E}_x\Big[\mbb{P}_{\tau^-(\ep^{n+1})}\big\{\tau^+_\ep< \tau^-_{\ep^{n+2}}\big\}\Big]\le\ep^n.
 \eeqnn
It follows from the Borel-Cantelli lemma that
$  \bP_x\{A_n \,\, \text{i.o.}\}=0$.
Therefore, by Proposition \ref{thmA} (ii), we have $\bP_x$-a.s. \,  $X_t<\ep$  for all $t$ large enough and the desired result follows.
\qed

\noindent{\it Proof of Proposition \ref{L0.2}.}
Observe that there is a constant $b'>0$ so that $\int^{b'}_0 z^2 \pi(\dd z)>0$.
Let
 \beqnn
m_0:=\sup_{y\in[a,b]}|\gamma_0 (y)|<\infty,~
m_1=\inf_{y\in[a,b]}\gamma_1 (y)
~\mbox{ and }~
m_2=\inf_{y\in[a,b]}\gamma_2 (y).
 \eeqnn
Since $m_1\vee m_2> 0$, there exists a large enough constant $c>0$ so that
 \beqnn
-cm_0+\frac{1}{2}c^2m_1+\frac{1}{2}c^2m_2\int^{b'}_0 z^2 \pi(\dd z)\ge 1.
 \eeqnn
Let $g$ be a convex function, i.e. $g''(y)\ge 0$, satisfying
$g(y)=\e^{cy}-\e^{ca}$ for $y\in [a, b+b']$ and $g''(y)= 0$ for $y> b+b'+1$.
Then by Lemma \ref{t4.1}, it is easy to see that
\beqnn
 \ar\ar
\int^\infty_0 (g(y+z)-g(y)-zg'(y))\pi(\dd z) \cr
 \ar=\ar
\int^{b+b'+1}_0 (g(y+z)-g(y)-zg'(y))\pi(\dd z) \cr
 \ar\ar\quad
+\int^{\infty}_{b+b'+1} (g(y+z)-g(y)-zg'(y))\pi(\dd z) \cr
 \ar\le\ar
\frac{1}{2}\sup_{y\in[a, b+b'+1]} g''(y)\int^{b+b'+1}_0 z^2 \pi(\dd z),
\eeqnn
which implies that condition (i) in Lemma \ref{lem1} is satisfied.
Observe that for any $y\in [a,b]$, we have
\beqnn
 \ar\ar
\int^\infty_0 (g(y+z)-g(y)-zg'(y))\pi(\dd z) \cr
 \ar\ar\quad\ge
\int^{b'}_0 (g(y+z)-g(y)-zg'(y))\pi(\dd z) \cr
\ar\ar\quad=
\e^{cy}\int^{b'}_0 (\e^{cz}-1-cz) \pi(\dd z)
\ge
\frac{1}{2}c^2\e^{cy} \int^{b'}_0 z^2 \pi(\dd z).
\eeqnn
Therefore, for any $y\in[a,b]$, we have
\beqnn
Lg(y)
\ar=\ar
\gamma_0(y)g'(y)+\frac12\gamma_1(y)g''(y) \cr
 \ar\ar
+\gamma_2(y)\int_0^\infty(g(y+z)-g(y)-zg'(y))\pi(\dd z) \cr
\ar\ge\ar
\gamma_0(y)c\e^{cy}+\frac{c^2}{2}\gamma_1(y)\e^{cy}+\frac{c^2}{2}\gamma_2(y)\e^{cy} \int^{b'}_0 z^2 \pi(\dd z) \cr
\ar\ge\ar
\e^{cy}\Big[-cm_0+\frac{c^2}{2}m_1+\frac{c^2}{2}m_2\int^{b'}_0 z^2 \pi(\dd z)\Big]
\ge
 \e^{Cy}\ge g(y).
\eeqnn
Applying Lemma \ref{lem1} yields $\bP_x\{\tau^+_b< \infty\}>0$.
\qed

\section{Proofs of the main results in Section \ref{main_result}}\label{proofs}
\setcounter{equation}{0}


Recall the definitions of $H_a$ and $G_a$ in \eqref{1.6} and \eqref{1.7}, respectively.  We now present the martingales we use to show the main results on extinction, explosion and coming down from infinity. It is remarkable that such a martingale is enough to show all the main results in this paper. Some other forms of martingales can only be used to prove  partial results.

\blemma\label{t3.3}
For $b>\ep>c>0$ let $T:=\tau^-_c\wedge\tau^+_b$.
Then the process  $X^{1-a}_{t\wedge T}\exp\Big\{\int_0^{t\wedge T}G_a(X_s)\dd s\Big\}$ is an $(\mcr{F}_t)$-martingale and
 \beqnn
\bE_\ep \bigg[X^{1-a}_T\exp\Big\{\int_0^TG_a(X_s)\dd s\Big\}\bigg]
\le\ep^{1-a}
 \eeqnn
for $a\neq1$.
\elemma
\proof  
 By It\^o's formula, we can see that
 \beqnn
X_t^{1-a}
 \ar=\ar
X_0^{1-a}
-\int_0^tG_a(X_s)X_s^{1-a}\dd s
+(1-a)\int_0^t\int_0^{\gamma_1(X_s)}X_s^{-a}W(\dd s,\dd u) \cr
 \ar\ar
+\int_0^t\int_0^\infty\int_0^{\gamma_2(X_{s-})}
\big[(X_{s-}+z)^{1-a}-X_{s-}^{1-a}\big]\tilde{N}(\dd s,\dd z, \dd u),
 \eeqnn
and it then follows from the integration by parts formula
(see e.g. Protter (2005, p. 68)) that
 \beqnn
 \ar\ar
X^{1-a}_{t}\exp\left\{\int_0^{t}G_a(X_s)\dd s\right\} \cr
 \ar=\ar
X_0^{1-a}
+\int^t_0X_s^{1-a}\exp\left\{\int^s_0 G_a(X_u)\dd u\right\}G_a(X_s)\dd s \cr
 \ar\ar
+\int^t_0\exp\left\{\int^s_0 G_a(X_u)\dd u\right\}\dd (X_s^{1-a})\cr
 \ar=\ar
X_0^{1-a}
+\mbox{local mart.}
 \eeqnn
Therefore,
 \beqlb\label{3.1}
t\mapsto X^{1-a}_{t\wedge T}\exp\Big\{\int_0^{t\wedge T}G_a(X_s)\dd s\Big\}
 \eeqlb
is a local martingale.
By Protter (2005, p. 38), \eqref{3.1} is a martingale if
 \beqlb\label{3.2}
\bE_\ep \bigg[\sup_{t\in[0,\delta]}
X^{1-a}_{t\wedge T}\exp\Big\{\int_0^{t\wedge T}G_a(X_s)\dd s\Big\}\bigg]<\infty.
 \eeqlb
for each $\delta>0$.
Observe that  for $0\leq t\leq \delta$
 \beqnn
\exp\Big\{\int_0^{t\wedge T} G_a(X_s)\dd s\Big\}
 \eeqnn
is    uniformly bounded from above   by a positive constant.
Then \eqref{3.2} is obvious for $a>1$.
In the following we consider the case $a<1$.
By the Burkholder-Davis-Gundy inequality,  we have
 \beqlb\label{3.3}
  \ar\ar
\bE_\ep \bigg[\sup_{t\in[0,\delta]}
\Big|\int_0^{t\wedge T}\int_0^{\gamma_1(X_{s-})}
W(\dd s,\dd u)\Big|^2\bigg] \cr
 \ar\le\ar
C\bE_\ep \bigg[\int_0^{\delta\wedge T}\dd s\int_0^{\gamma_1(X_{s-})} \dd u\bigg]
\le C\delta\sup_{x\in[0,b]}\gamma_1(x)
 \eeqlb
and
 \beqlb\label{3.4}
 \ar\ar
\bE_\ep \bigg[\sup_{t\in[0,\delta]}\Big|\int_0^{t\wedge T}\int_{0}^1\int_0^{\gamma_2(X_{s-})} z\tilde{N}(\dd s, \dd z, \dd u)\Big|^2\bigg] \cr
 \ar\ar\quad\le
C\bE_\ep \bigg[\int_0^{\delta\wedge T}\dd s\int_0^1z^2\pi(\dd z)\int_0^{\gamma_2(X_{s-})} \dd u\bigg] \cr
 \ar\ar\quad\le
C\delta\int_0^1z^2\pi(\dd z)\sup_{x\in[0,b]}\gamma_2(x).
 \eeqlb
Observe that
 \beqlb\label{3.5}
 \ar\ar
\bE_\ep \bigg[\sup_{t\in[0,\delta]}\Big|\int_0^{t\wedge T}\int_1^\infty\int_0^{\gamma_2(X_{s-})} z\tilde{N}(\dd s, \dd z, \dd u)\Big|\bigg] \cr
 \ar\le\ar
\bE_\ep \bigg[\sup_{t\in[0,\delta]}\Big|\int_0^{t\wedge T}\int_{1}^\infty\int_0^{\gamma_2(X_{s-})} zN(\dd s, \dd z, \dd u)\Big|\bigg] \cr
 \ar\ar
+\bE_\ep \bigg[\sup_{t\in[0,\delta]}\Big|\int_0^{t\wedge T}\dd s\int_1^\infty
z\pi(\dd z)\int_0^{\gamma_2(X_{s-})}\dd u\Big|\bigg]
\cr
 \ar\le\ar
2\delta\int_1^\infty z\pi(\dd z)\sup_{x\in[0,b]}\gamma_2(x).
 \eeqlb
It then follows from \eqref{sdeB} and \eqref{3.3}--\eqref{3.5} that
$\bE_\ep[\sup_{t\in[0,\delta]}
X_{t\wedge T}]<\infty$, which implies
\eqref{3.2}.
Now by Fatou's lemma, we get
 \beqnn
\bE_\ep \bigg[X^{1-a}_T\exp\Big\{\int_0^TG_a(X_s)\dd s\Big\}\bigg]
 \ar=\ar
\bE_\ep \bigg[\lim_{t\to\infty}X^{1-a}_{t\wedge T}\exp\Big\{\int_0^{t\wedge T}G_a(X_s)\dd s\Big\}\bigg] \cr
 \ar\le\ar
\lim_{t\to\infty}\bE_\ep \bigg[X^{1-a}_{t\wedge T}\exp\Big\{\int_0^{t\wedge T}G_a(X_s)\dd s\Big\}\bigg]\cr
\ar=\ar\ep^{1-a},
 \eeqnn
which finishes the proof.
\qed

\noindent{\it Proof of Theorem \ref{t2.5}.}
(i) In the present proof for $n=2,3,\ldots$, let $T_n:=\tau^-(\ep^n)\wedge\tau^+_b$ for small enough $0<\ep<b$. It follows from Lemma \ref{t3.3} that
 \beqnn
\ep^{1-a}
 \ar\geq\ar
\bE_{\ep} \Big[X^{1-a}_{\tau^-(\ep^{n})\wedge\tau^+(b)}\exp\big\{-(\ln \ep^{-n})^r  (\tau^-(\ep^{n})\wedge\tau^+(b) \big\}\Big]\cr
 \ar\geq\ar
\bE_{\ep} \Big[X^{1-a}_{\tau^-(\ep^{n})}\exp\big\{-(\ln \ep^{-n})^r d_n  \big\}1_{\{\tau^-(\ep^{n})<\tau^+_b\wedge d_n\}}\Big] \cr
 \ar=\ar
\ep^{(1-a)n}  \exp\{\ln \ep^{n(a-1)/2}\}
\bP_\ep\big\{\tau^-(\ep^{n})<\tau^+_b\wedge d_n \big\},
 \eeqnn
where
\beqnn
d_n:=\frac{\ln \ep^{n(a-1)/2} }{-(\ln \ep^{-n})^r}
=\frac{n(a-1)/2\ln\ep^{-1}}{n^r(\ln\ep^{-1})^r}\goto\infty
\eeqnn
as $n\goto\infty$. Then
 \beqnn
\bP_{\ep}\left\{\tau^-(\ep^{n})<\tau^+_b\wedge d_n \right\}\leq  \ep^{(a-1)(n-2)/2}.
 \eeqnn

By the Borel-Cantelli Lemma we
have
 \beqlb\label{survive_b}
\bP_\ep \Big\{\tau^-{(\ep^{n})}
<\tau^+_b\wedge d_n ~~ \text{i.o.} \Big\}=0.
 \eeqlb
Then $\bP_\ep$-a.s.,
 \beqnn
\tau^-{(\ep^{n})}
\ge\tau^+_b\wedge d_n
 \eeqnn
for $n$ large enough.

Now if there are infinitely many $n$ so that
 \beqlb\label{survive_a}
\tau^-{(\ep^{n})}\ge d_n,
 \eeqlb
then we have $\tau^-_0=\infty$; on the other hand,
 if \eqref{survive_a} holds for at most finitely many $n$, then by (\ref{survive_b}) we have
 $\tau_b^+<\tau^-{(\ep^{n})}$  for all $n$ large enough.
Combining these two cases,
\begin{equation}\label{survive_c}
\bP_\ep\{ \tau^-_0=\infty \,\, \text{or}\,\, \tau^+_b<\tau^-_0   \}=1.
\end{equation}
It follows from the Markov property and lack of negative jumps for $X$ that  if $\bE_\ep[\e^{-\lambda\tau^-_0 }; \, \tau^-_0<\infty]>0$ for $\lambda>0$, then
\begin{equation*}
\begin{split}
\bE_\ep[\e^{-\lambda\tau^-_0 }; \, \tau^-_0<\infty]&=\bE_\ep[\e^{-\lambda\tau^-_0 }; \, \tau^+_b<\tau^-_0<\infty]\\
&\leq \bE_\ep[\e^{-\lambda\tau^+_b }; \, \tau^+_b<\tau^-_0]\bE_b[\e^{-\lambda\tau^-_\ep }; \, \tau^-_\ep<\infty]\\
&\quad\quad\quad\times\bE_\ep[\e^{-\lambda\tau^-_0 }; \, \tau^-_0<\infty]\\
&<\bE_\ep[\e^{-\lambda\tau^-_0 }; \, \tau^-_0<\infty],
\end{split}
\end{equation*}
where we need (\ref{survive_c}) for the first equation.
Therefore, $\bE_\ep[\e^{-\lambda\tau^-_0 }; \, \tau^-_0<\infty]=0$ and consequently, $\bP_\ep\{\tau^-_0<\infty\}=0$.

One can also find  similar arguments in the proof of
Theorem 4.2.2 in Le (2014) and the proof of Theorem 2.8(2) in Le and Pardoux (2015).

(ii) Given  $0<\de< \frac{1}{3-2a}$,  consider the martingale
 \beqnn
X^{1-a}_{t\wedge T}\exp\Big\{\int_0^{t\wedge T}G_{a}(X_s)\dd s\Big\}
 \eeqnn
for $T=\tau^-(\ep^{1+\delta})\wedge\tau^+(\ep^{1-\delta})  $.
By Lemma \ref{t3.3},
 \beqnn
\ep^{1-a}
 \ar\geq\ar
\bE_{\ep}\Big[ X^{1-a}
_{\tau^+(\ep^{1-\delta})}\exp\Big\{\int_0^{\tau^+(\ep^{1-\delta})}G_{a}(X_s)\dd s\Big\}1_{\{ \tau^+(\ep^{1-\delta})<\tau^-(\ep^{1+\delta})\}} \Big] \cr
 \ar\geq\ar
\ep^{(1-a)(1-\delta) }
\bP_{\ep}\big\{ \tau^+(\ep^{1-\delta})<\tau^-(\ep^{1+\delta})\big\}.
 \eeqnn
Then
\begin{equation}\label{extinctA}
\bP_{\ep}\big\{\tau^+(\ep^{1-\de})<\tau^-(\ep^{1+\de})\big\}\leq \ep^{(1-a)\delta}.
\end{equation}

Similarly,
\[\ep^{1-a}\geq \bE_{\ep}\Big[ X^{1-a}_t\exp\Big\{\int_0^t G_{a}(X_s)\dd s\Big\}1_{\{ \tau^+(\ep^{1-\delta})=\tau^-(\ep^{1+\delta})=\infty\}} \Big].\]
Letting $t\goto\infty$ we have
\begin{equation}\label{extinctB}
\bP_{\ep}\big\{\tau^+(\ep^{1-\de})=\tau^-(\ep^{1+\de})=\infty\big\}=0.
\end{equation}

By Lemma \ref{t3.3}   again, for $t(\ep):=[-(1-\de)\ln\ep]^{1-r}$ we have
 \beqnn
\ep^{1-a}
 \ar\geq\ar
\bE_{\ep}\Big[ X^{1-a}_{\tau^-(\ep^{1+\delta})}\exp\Big\{
\int_0^{\tau^-(\ep^{1+\delta})} G_{a}(X_s)\dd s\Big\} 1_{\{t(\ep)<\tau^-(\ep^{1+\delta})<\tau^+(\ep^{1-\delta}) \}}\Big] \cr
 \ar\geq\ar
\ep^{(1-a)(1+\delta)}\bE_{\ep}\Big[ \e^{ [-(1-\de)\ln\ep]^r t(\ep)}
1_{\{t(\ep)<\tau^-(\ep^{1+\delta})<\tau^+(\ep^{1-\delta}) \}} \Big] \cr
 \ar=\ar
\ep^{(1-a)(1+\delta)} \ep^{-(1-\de)}\bP_{\ep}\Big\{t(\ep)<\tau^-(\ep^{1+\delta})<\tau^+(\ep^{1-\delta}) \Big\}.
 \eeqnn
Then
 \beqlb\label{extinctC}
\quad\bP_{\ep}\Big\{t(\ep)<\tau^-(\ep^{1+\delta})<\tau^+(\ep^{1-\delta})\Big\}
\leq  \ep^{(a-1)\de}\ep^{1-\de} = \ep^{1+(a-2)\de}.
 \eeqlb
Combining (\ref{extinctA}), (\ref{extinctB}) and (\ref{extinctC}),  we have
\beqnn  \bP_{\ep}\{\tau^-(\ep^{1+\delta})>t(\ep)\}\leq \ep^{1+(a-2)\de}+\ep^{(1-a)\delta}<2\ep^{(1-a)\delta}.\eeqnn

By the strong Markov property and lack of negative jumps for process $X$,
 \beqnn
 \ar\ar
\bP_\ep\Big\{\bigcap_{n=0}^m\big\{\tau^-(\ep^{(1+\de)^n})<\infty, \, \tau^-(\ep^{(1+\de)^{n+1}})\circ\theta_{\tau^-(\ep^{(1+\de)^n})}
\leq t(\ep^{(1+\de)^n}) \big\}\Big\} \cr
 \ar\ar\quad=
\prod_{n=0}^m \bP_{\ep^{(1+\de)^n}}\Big\{\tau^-(\ep^{(1+\de)^{n+1}})\leq t(\ep^{(1+\de)^n}) \Big\}
\ge\prod_{n=0}^m \Big[1-2\ep^{(1+\de)^n(1-a)\de}\Big] \cr
\ar\ar\quad\ge\prod_{n=0}^m \e^{-4\ep^{(1+\de)^n (1-a)\de}}
\ge \e^{-8\ep^{(1-a)\delta}}.
 \eeqnn
Letting $m\goto \infty$ we have
 \beqnn
 \ar\ar
\bP_\ep\Big\{\bigcap_{n=0}^\infty\Big\{\tau^-(\ep^{(1+\de)^n})<\infty,  \cr
 \ar\ar\qquad\qquad
\tau^-(\ep^{(1+\de)^{n+1}})\circ\theta_{\tau^-(\ep^{(1+\de)^n})}
\leq t(\ep^{(1+\de)^n}) \Big\}\Big\}\ge \e^{-8\ep^{(1-a)\delta}}.
 \eeqnn

Since under $\bP_\ep$,
\beqnn  \tau^-_{0-}=\sum_{n=0}^\infty \tau^-(\ep^{(1+\de)^{n+1}})\circ\theta_{\tau^-(\ep^{(1+\de)^n})}, \eeqnn
then
\beqnn  \bP_\ep\Big\{\tau^-_{0-}\leq \sum_{n=0}^\infty t(\ep^{(1+\de)^n}) \Big\}\ge\e^{-8\ep^{(1-a)\delta}}.\eeqnn
Notice that for $\ep_n:=\ep^{(1+\de)^n}$,
 \beqnn
\sum_{n=1}^\infty t(\ep_n)=\sum_{n=1}^\infty \big[(\de-1)\ln\ep_n\big]^{1-r}
=\sum_{n=1}^\infty \big[(1+\de)^n(\de-1)\ln\ep\big]^{1-r}<\infty,
 \eeqnn
we thus have
 \beqnn
\bP_\ep\left\{\tau^-_{0-}< \infty \right\}\ge\e^{-8\ep^{(1-a)\delta}}.
 \eeqnn
By the definition of solution to SDE (\ref{sdeB}) at the  beginning of Section 2,  we have
 \beqlb\label{2.12}
\bP_\ep\left\{\tau^-_0=\tau^-_{0-}<\infty \right\}\ge\e^{-8\ep^{(1-a)\delta}},
 \eeqlb
which finishes the proof.
\qed

\noindent{\it Proof of Theorem \ref{t3.1}.}
(i) In the present proof, for small enough $b^{-1}$ and $\ep$ satisfying  $0<b<\ep^{-1}$ and for $n=2,3,\ldots$, let $T_n:=\tau^-({b})\wedge\tau^+(\ep^{-n})$.
By Lemma \ref{t3.3} we have
\beqnn
 \ep^{a-1}
 \ar\geq\ar
\bE_{\ep^{-1}}\Big[ X^{1-a}_{\tau^+(\ep^{-n})\wedge\tau^-({b})}
\exp\Big\{-(\ln \ep^{-n})^r
(\tau^+(\ep^{-n})\wedge\tau^-({b})) \Big\}\Big] \cr
 \ar\geq\ar
\bE_{\ep^{-1}} \Big[X^{1-a}_{\tau^+(\ep^{-n})}
\exp\big\{-(\ln \ep^{-n})^r d_n \big\}1_{\{\tau^+(\ep^{-n})<\tau^-({b})\wedge d_n \}}\Big] \cr
 \ar\geq\ar
\ep^{(a-1)n} \bE_{\ep^{-1}} \Big[\exp\{\ln \ep^{(1-a)n/2}\} 1_{\{\tau^+(\ep^{-n})<\tau^-({b})\wedge d_n \}}\Big]
 \eeqnn
for $b$ and $\ep^{-1}$ large enough, where
 \beqnn
d_n :=\frac{(1-a)n\ln\ep^{-1}}{2(\ln\ep^{-n})^r}
=\frac{(1-a)n^{1-r}}{2}(\ln\ep^{-1})^{1-r}\goto\infty
 \eeqnn
as $n\goto\infty$. Then
\beqnn  \bP_{\ep^{-1}}\Big\{\tau^+_{\ep^{-n}}<\tau^-({b})\wedge d_n \Big\}\leq  \ep^{(1-a)(n-2)/2}\eeqnn
for large enough $b$ and $\ep^{-1}$.
The desired result of part (i) then follows from an argument similar to that in the proof for Theorem \ref{t2.5} (i).

(ii)
Taking $T:=\tau^-(\ep^{-1+\delta})\wedge\tau^+(\ep^{-1-\delta})$ in
Lemma \ref{t3.3}, we get
 \beqnn
\ep^{a-1}
 \ar\geq\ar
\bE_{\ep^{-1}}\Big[ X^{1-a}
_{\tau^-(\ep^{-1+\delta})}\exp\Big\{\int_0^{\tau^-(\ep^{-1+\delta})}G_{a}(X_s)\dd s\Big\} \cr
 \ar\ar\qqquad\qqquad\qquad~\times
1_{\{ \tau^-(\ep^{-1+\delta})<\tau^+(\ep^{-1-\delta})\}} \Big]\cr
 \ar\geq\ar
\ep^{(a-1)(1-\delta) }
\bP_{\ep^{-1}}\Big\{ \tau^-(\ep^{-1+\delta})<\tau^+(\ep^{-1-\delta})\Big\}.
 \eeqnn
Then
\begin{equation}\label{explosion_a'}
\bP_{\ep^{-1}}\Big\{\tau^-(\ep^{-(1-\de)})<\tau^+(\ep^{-(1+\de)})\Big\}
\leq \ep^{(a-1)\delta}.
\end{equation}
  Similarly,
\begin{equation*}
\begin{split}
\ep^{a-1}
&\geq\bE_{\ep^{-1}}\Big[ X^{1-a}_{n}\exp\Big\{\int_0^{n}G_{a}(X_s)\dd s\Big\}  1_{\{ \tau^-(\ep^{-1+\delta})\wedge\tau^+(\ep^{-1-\delta})> n\}} \Big]\\
&\geq  \ep^{(1+\delta)(a-1)}\e^{n \ep^{-1+\delta}} \bP_{\ep^{-1}}\{\tau^-(\ep^{-1+\delta})\wedge\tau^+(\ep^{-1-\delta})> n\}.
\end{split}
\end{equation*}
Letting $n\goto\infty$ we have
\begin{equation}\label{explosion_b'}
\bP_{\ep^{-1}}\Big\{\tau^-(\ep^{-(1-\de)})=\tau^+(\ep^{-(1+\de)})=\infty\Big\}=0.
\end{equation}

Let $t(y):=(\ln y^{1-\de})^{1-r}$ for $y>1$ and small enough $\delta$.
With $T$ replaced by
$t(\ep^{-1})\wedge\tau^-(\ep^{-1+\delta})\wedge\tau^+(\ep^{-1-\delta})$,
similar to the above argument we get
 \beqnn
\ep^{a-1}
 \ar\geq\ar
\bE_{\ep^{-1}}\Big[ X^{1-a}_{t(\ep^{-1})}\exp\Big\{
\int_0^{t(\ep^{-1})} G_{a}(X_s)\dd s\Big\}  \cr
 \ar\ar\qqquad\qquad\times1_{\{t(\ep^{-1})<\tau^+(\ep^{-(1+\delta)})<\tau^-(\ep^{-(1-\delta)}) \}}\Big] \cr
 \ar\geq\ar
\ep^{(a-1)(1+\delta)}\bE_{\ep^{-1}}\Big[\e^{((\de-1)\ln\ep)^r t(\ep^{-1}) } 1_{\{t(\ep^{-1})<\tau^+(\ep^{-(1+\delta)})<\tau^-(\ep^{-(1-\delta)}) \}}
 \Big] \cr
 \ar=\ar
\ep^{(a-1)(1+\delta)}\bE_{\ep^{-1}}\left[\e^{(\de-1)\ln\ep} 1_{\{t(\ep^{-1})<\tau^+(\ep^{-(1+\delta)})<\tau^-(\ep^{-(1-\delta)}) \}}
 \right].
 \eeqnn
Then
 \begin{equation}\label{explosion_c'}
 \begin{split}
&\bP_{\ep^{-1}}\Big\{t(\ep^{-1})<\tau^+(\ep^{-(1+\delta)})
<\tau^-(\ep^{-(1-\delta)})\Big\}\\
&\leq  \ep^{(1-a)\delta} \e^{-(\de-1)\ln\ep}= \ep^{1-a\de}.
\end{split}
 \end{equation}
Combining (\ref{explosion_a'}), (\ref{explosion_b'}) and (\ref{explosion_c'}), we have
 \begin{equation}\label{explosion_c}
\bP_{\ep^{-1}}\big\{\tau^+(\ep^{-(1+\delta)})>t(\ep^{-1})\big\}\leq 2\ep^{(a-1)\delta}.
 \end{equation}

Write $\tilde{\tau}_0:=0$ and $\tilde{\tau}_{n+1}:=\tilde{\tau}^+((X_{\tilde{\tau}_n}\vee 1)^{1+\de})\circ\tilde{\tau}_n$ for $ n=0,1,2, \ldots$
with the convention $X_{\infty}=0$.
 Notice that $X$ allows possible positive jumps, and under $\bP_{\ep^{-1}}$ for $n\geq 1$, $X_{\tilde{\tau}_n}\geq\ep^{-(1+\de)^n}$ if $\tilde{\tau}_n<\infty$.

Observe that under $\bP_{\ep^{-1}}$, if $\tilde{\tau}_n<\infty$ for all $n\geq 1$, then
 \beqnn
\sum_{n=1}^\infty t(X_{\tilde{\tau}_n})
\leq\sum_{n=1}^\infty \big[\ln \ep^{-(1+\de)^n(1-\de)} \big]^{1-r}
=\sum_{n=1}^\infty \big[(1+\de)^n(\de-1)\ln\ep\big]^{1-r}<\infty.
 \eeqnn
By the strong Markov property and estimate (\ref{explosion_c}), we can show that
 \beqnn
\bP_{\ep^{-1}}\{\tilde{\tau}^+_\infty<\infty \}
 \ar=\ar
\bP_{\ep^{-1}}\big\{\lim_n\tilde{\tau}_n<\infty\big\} \cr
 \ar\geq\ar
\bP_{\ep^{-1}}\big\{ \tilde{\tau}_{n+1}<t(X_{\tilde{\tau}_{n}}) \quad\text {for all}\quad n\geq 1 \big\} \cr
 \ar\geq\ar
\prod_{n=1}^\infty \Big[1-2\ep^{(a-1)\delta(1+\de)^n}\Big] >0.
 \eeqnn
This finishes the proof.
\qed



\noindent{\it Proof of Theorem \ref{comingdown}.}
To show part (i), for any constants $d>0$ and $b>0$ such that (\ref{condition_stay}) holds for all $u>b$,
for any $0<\ep<b^{-1}$, we have
\begin{equation}\label{stay_infinity}
\begin{split}
&\bP_{\ep^{-1}}\{\tau^-_b<d\}\\
&\leq \bP_{\ep^{-1}}\{\tau^-_b<d\wedge \tau^+(\ep^{-2})\}\\
&\quad+\sum_{n=1}^\infty \bP_{\ep^{-1}}\left\{\tau^+(\ep^{-2^n})<\tau^-_b<d, \, \sup_{0\leq s\leq\tau^-_b}X_s\in[\ep^{-2^n},  \ep^{-2^{n+1}})  \right\}\\
&\leq \bP_{\ep^{-1}}\{\tau^-_b<d\wedge \tau^+(\ep^{-2})\}\\
&\quad+\sum_{n=1}^\infty \bP_{\ep^{-1}}\left\{\tau^+(\ep^{-2^n})<\tau^-_b,  \right. \\
&\qquad\qquad\left.\, \tau^-_b\circ\theta(\tau^+(\ep^{-2^n}))
<d\wedge \tau^+(\ep^{-2^{n+1}})\circ\theta(\tau^+(\ep^{-2^n}))\right\}\\
&\leq \bP_{\ep^{-1}}\{\tau^-_b<d\wedge \tau^+(\ep^{-2})\}\\
&\quad+\sum_{n=1}^\infty \bE_{\ep^{-1}}\left\{ 1_{\{\tau^+(\ep^{-2^n})<\tau^-_b\}}\bP_{X_{\tau^+(\ep^{-2^n})}}\left\{\tau^-_b<d\wedge \tau^+(\ep^{-2^{n+1}}) \right\}\right\}.
\end{split}
\end{equation}

 By Lemma \ref{t3.3},
for $\ep^{-2^n}\leq x<\ep^{-2^{n+1}}$,
$T:=\tau^-_b\wedge \tau^+(\ep^{-2^{n+1}})\wedge d$ and $n=0, 1,2,\ldots$,
\begin{equation*}
\begin{split}
x^{1-a}
&\geq \bE_x X^{1-a}_{T}\exp\left\{\int_0^T G_a(X_s)\dd s\right\}\\
&\geq \bE_x X^{1-a}_{T}\exp\left\{-\int_0^T (\ln(X_s))^r \dd s\right\}\\
&\geq b^{1-a} \bE_x \left[\exp\{-d(\ln \ep^{-2^{n+1}})^r\}; \,   \tau^-_b< \tau^+(\ep^{-2^{n+1}})\wedge d\right]\\
&= b^{1-a} \bE_x \left[\exp\{-d(2^{n+1}\ln \ep^{-1})^r\}; \,   \tau^-_b< \tau^+(\ep^{-2^{n+1}})\wedge d\right].\\
\end{split}
\end{equation*}
Then
\begin{equation*}
\begin{split}
\bP_x \{\tau^-_b< \tau^+(\ep^{-2^{n+1}})\wedge d\}&\leq  b^{a-1}\exp\{(1-a)2^n \ln \ep^{-1}+d(2^{n+1}\ln \ep^{-1})^r \}\\
&\leq b^{a-1}\e^{(1-a)2^{n-1} \ln \ep^{-1}}= b^{a-1}\ep^{(a-1)2^{n-1}}.
\end{split}
\end{equation*}
for all small enough $\ep>0$.
It follows from (\ref{stay_infinity}) and the strong Markov property that
\[\bP_{\ep^{-1}}\{\tau^-_b<d\}\leq  b^{a-1}\sum_{n=0}^\infty \ep^{(a-1)2^{n-1}},\]
which goes to $0$ as $\ep\goto 0+$. Since $b>0$ and $d>0$ are arbitrary, the process $X$ thus  stays at infinity.

We now proceed to show the  part (ii).
Write $t(x):=(1+\de)^r(\ln x)^{1-r}$ for $x>1$. Then  by Lemma \ref{t3.3}, for large $x$,
\begin{equation*}
\begin{split}
&x^{1-a}\\
&\geq \bE_x \left\{X^{1-a}_{\tau^+(x^{(1+\de)})}
\exp\left\{\int_0^{\tau^+(x^{(1+\de)})}G_a(X_s)\dd s \right\}1_{\{\tau^-(x^{(1+\de)^{-1}})>\tau^+(x^{1+\de})\}}\right\} \\
&\geq x^{(1-a)(1+\de)}\bP_x \{\tau^-(x^{(1+\de)^{-1}})>\tau^+(x^{1+\de})\}.
\end{split}
\end{equation*}
Then
\begin{equation}\label{bound1}
\bP_x \{\tau^-(x^{(1+\de)^{-1}})>\tau^+(x^{1+\de})\}\leq x^{-\de(1-a)}.
\end{equation}
By condition (\ref{comingdown_a}) we also have
\begin{equation*}
\begin{split}
&x^{1-a}\\
&\geq \bE_x \left\{X^{1-a}_{\tau^-(x^{(1+\de)^{-1}})}
\exp\left\{\int_0^{\tau^-(x^{(1+\de)^{-1}})}G_a(X_s)\dd s \right\}\right\}\\
&\quad\quad\quad \times 1_{\{t(x)<\tau^-(x^{(1+\de)^{-1}})<\tau^+(x^{1+\de})\}}\\
&\geq x^{(1-a)/(1+\de)}\bE_x\left\{\exp\left\{\int_0^{t(x)}(\ln X_s)^r\dd s \right\}1_{\{t(x)<\tau^-(x^{(1+\de)^{-1}})<\tau^+(x^{1+\de})\}}\right\} \\
&\geq x^{(1-a)/(1+\de)} \e^{(1+\de)^{-r}(\ln x)^rt(x)}\bP_x\{t(x)<\tau^-(x^{(1+\de)^{-1}})<\tau^+(x^{1+\de})\}\\
&= x^{(1-a)/(1+\de)} x\bP_x\{t(x)<\tau^-(x^{(1+\de)^{-1}})<\tau^+(x^{1+\de})\}.\\
\end{split}
\end{equation*}
It follows that
\begin{equation}\label{bound2}
\bP_x\{t(x)<\tau^-(x^{(1+\de)^{-1}})<\tau^+(x^{1+\de})\}\leq x^{\frac{\de(1-a)}{1+\de}-1}= x^{-(1+\de a)/(1+\de)}.
\end{equation}
Combining (\ref{bound1}) and (\ref{bound2}) we have
\[\bP_x\{t(x)<\tau^-(x^{(1+\de)^{-1}})\}\leq  x^{-(1+\de a)/(1+\de)}+x^{-\de(1-a)}\leq 2x^{-\de(1-a)}\]
for small enough $\de>0$. Then for $b\equiv b(\de)$ large enough, by the strong Markov property
\begin{equation*}
\begin{split}
&\bP_{b^{(1+\de)^m}}\left\{\cap_{n=1}^m \left\{\tau^-(b^{(1+\de)^n})<\infty,  \tau^-(b^{(1+\de)^{n-1}})\circ \theta_{\tau^-(b^{(1+\de)^n})}\leq t(b^{(1+\de)^n}) \right\}  \right\}\\
&=\prod_{n=1}^m \bP_{b^{(1+\de)^n}}\left\{ \tau^-(b^{(1+\de)^{n-1}}) \leq t(b^{(1+\de)^n})\right\} \\
&\geq \prod_{n=1}^m \left(1-2b^{-\de(1-a)(1+\de)^{n-1}}  \right)\geq \prod_{n=1}^m \e^{-4b^{-\de(1- a)(1+\de)^{n-1}}}\\
&=\e^{-4\sum_{n=1}^m b^{-\de(1- a)(1+\de)^{n-1}}}\geq \e^{-8b^{-\de(1- a)}}.
\end{split}
\end{equation*}
Let $m\goto\infty$. Then
\begin{equation}\label{comedown_a}
\begin{split}
&\lim_{x\to\infty}\bP_x \left\{\tau^-_b \leq \sum_{n=1}^\infty t(b^{(1+\de)^n})=(1+\de)^r(\ln b)^{1-r}\sum_{n=1}^\infty (1+\de)^{(1-r)n}<\infty\right\}\\
&\geq \e^{-8b^{-\de(1- a)}}
\end{split}
\end{equation}
for $r>1$. Letting $b\goto\infty$ in (\ref{comedown_a}), we obtain the limit (\ref{comedown_def}) and
the process $X$ comes down from infinity.
\qed


\section*{Acknowledgements}
The authors are grateful to an anonymous referee and an associate editor for very helpful comments.
Pei-Sen Li thanks Concordia University where part of the work on this paper was carried out during his visit as a postdoctoral fellow.


\end{document}